\documentclass[11pt, reqno]{amsart}
\usepackage{subfig, enumerate, tikz-cd}
\usepackage{bookmark}
\usepackage{multirow}
\usepackage{filecontents}
\usepackage[shortlabels]{enumitem}
\usepackage[section] {placeins}
\usepackage{cases}
\usepackage{amscd,amsfonts,amssymb,amsmath,tikz-cd}
\usepackage{graphicx}
\usepackage{bm}
\usepackage{hyperref}
\usepackage{epsfig}
\usepackage{mathtools}
\usepackage{float}
\usepackage{color}
\usepackage[normalem]{ulem}
\usepackage[sort,nocompress]{cite}
\usepackage{array}
\newtheorem{theorem}{Theorem}[section]

\newtheorem{proposition}[theorem]{Proposition}
\newtheorem{corollary}[theorem]{Corollary}

\newtheorem{lemma}[theorem]{Lemma}

\theoremstyle{definition}
\newtheorem{definition}{Definition}[section]
\newtheorem{example}{Example}

\newtheorem*{examples}{Examples}

\newtheorem{remark}{Remark}[section]

\newcommand{\im}{\operatorname{Im}}

\newcommand{\YB}{\operatorname{YB}}
\newcommand{\TYB}{\operatorname{TYB}}

\newcommand{\TD}{\operatorname{TD}}
\newcommand{\TBQ}{\operatorname{TBQ}}

\newcommand{\T}{\operatorname{T}}

\newcommand{\Hom}{\operatorname{Hom}}

\newcommand{\id}{\mathrm{id}}

\newcommand{\overbar}[1]{\mkern 1.5mu\overline{\mkern-1.5mu#1\mkern-1.5mu}\mkern 1.5mu}

\usepackage[autostyle]{csquotes}

\usepackage{calc}

\begin{document}
\title{Twisted Yang-Baxter sets, cohomology theory, and application to knots}
\author{Mohamed Elhamdadi}
\author{Manpreet Singh}
\address{University of South Florida \\
Tampa, FL, 33620}
\email{emohamed@usf.edu}
\email{manpreet.singh2@fulbrightmail.org\\ manpreet.math23@gmail.com}


\begin{abstract}
We introduce twisted set-theoretic Yang-Baxter solutions and develop an associated cohomology theory, which extends the standard cohomology theory of Yang-Baxter solutions. By employing cocycles of twisted biquandles along with Alexander numbering, we construct state-sum invariants for knots and knotted surfaces. As an application, we use our approach to distinguish the $2$-twist spun trefoil from its reverse orientation, in line with prior findings.
\end{abstract}

\maketitle
\section{Introduction}

The Yang-Baxter equation, independently discovered by Yang \cite{MR261870} and Baxter \cite{MR290733}, has been pivotal in the fields of quantum groups, braided categories, and knot theory. A Yang-Baxter set consists of a pair $(X,R)$, where $X$ is a set and $R: X \times X \to X \times X$ is an invertible map that satisfies:
\[
(R \times \id_X ) (\id_X \times R) (R \times \id_X)= (\id_X \times R) (R \times \id_X) (\id_X \times R)
\]
with $\id_X: X \to X$ being the identity map in $X$.
Solutions to the Yang-Baxter equation enable the construction of knot invariants \cite{MR939474, MR990215}, such as the Jones polynomial \cite{MR908150}. Algebraic structures like racks, quandles \cite{MR638121, MR672410}, and biquandles \cite{MR1973514, MR1721925} also serve as solutions to the Yang-Baxter equation.

In \cite{MR1990571}, a comprehensive homological theory of quandles was established, utilizing quandle cocycles to create state-sum invariants for codimension $2$-embeddings. Advancing this work, \cite{MR2128041} develops a homology theory for the Yang-Baxter sets through the application of Yang-Baxter coloring of cubical complexes, and invariants of knots are defined by use of biquandle cocycles. Additionally, \cite{MR3381331} proposes a graphical approach to define a homology theory for Yang-Baxter sets. Remarkably, \cite{MR3835755} proves that both the cubical complex and graphical approaches give rise to the same homology theory.

In \cite{MR1990571}, a quandle cocycle was pivotal in proving that the $2$-twist spun trefoil knot is non-invertible. The remarkable potential of quandle cocycle invariants is further highlighted in \cite{MR1984465}, where they were used to determine the exact triple point number of the $2$-twist spun trefoil knot, which is four. Additionally, \cite{MR2119028} proves the existence of infinitely many non-invertible knotted surfaces by employing quandle cocycle invariants associated with Mochizuki’s $3$-cocycles \cite{MR1960136}. These findings underscore the profound implications of quandle cocycle invariants in knot theory.

In this article, we introduce a notion of twisted set-theoretic Yang-Baxter solution. A twisted set-theoretic Yang-Baxter solution consists of a triplet $(X,f,R)$, where $(X,R)$ denotes a Yang-Baxter set and $f:X \to X$ is an automorphism of $(X,R)$. We introduce a (co)homology theory for this concept using a graphical approach. Subsequently, we apply this theory to define cohomology for twisted biquandles. We then employ cocycles of twisted biquandles, along with Alexander numbering, to formulate a state-sum invariant for knots and knotted surfaces. We compute a state-sum invariant of $2$-twist spun trefoil and of its reverse orientation.

Throughout the paper we refer to knots and links by the generic term “knots”.

The article is organized as follows.  In Section~\ref{prelim} we provide an overview of the Yang-Baxter sets and Yang-Baxter homology theory, focusing on the graphical approach \cite{MR3381331}.  Section~\ref{sec:twisted_yang_baxter_solutions} introduces {\it twisted set-theoretic Yang-Baxter solutions}. In Section \ref{sec:homology_theory_for_twisted_yang_baxter_set} we introduce a (co)homology theory for twisted Yang-Baxter sets and twisted biquandles. Section \ref{sec:Calcualtions_of_homology} discusses the computation of second homology groups for certain twisted quandles, highlighting differences from their counterparts in quandles. Section \ref{sec:extensions} delves into the extension theory of twisted Yang-Baxter sets using cocycles. In Section \ref{sec:twisted_biquandle_cocycle_invariants_of_classical_knots}, we introduce a state-sum knot invariant by utilizing $2$-cocycles for twisted biquandles alongside Alexander numbering of knots. A similar approach is taken in Section \ref{sec:twisted_biquandle_cocycle_invariants_of_knotted_surfaces} to define state-sum invariants for knotted surfaces. Section \ref{sec:computations_using_braid_charts} describes a method for computing state-sum invariants for knotted surfaces using braid charts. In Section \ref{sec:concluding_remarks}, we note that one can analogously define state sum invariants for knots on compact oriented surfaces and broken surface diagrams in orientable compact $3$-manifolds.

\section{Preliminary}\label{prelim}

Let $X$ be a non-empty set and $R: X \times X \to X \times X$ an invertible map, satisfying the following {\it set-theoretic Yang-Baxter equation:}
\[
 (R\times \id_X) (\id_X \times R) (R \times \id_X)=(\id_X \times R) (R \times \id_X) (\id_X \times R),
\]
where $\id_X$ denotes the identity map on $X$. Then $R$ is termed a {\it set-theoretic Yang-Baxter operator}, and the pair $(X,R)$ is said to be a {\it set-theoretic Yang-Baxter solution}. Sometimes, for brevity, we call $(X,R)$ a {\it Yang-Baxter set}.

\par

We denote the components of $R$ by $R_1$ and $R_2$, that is, $R(x_1,x_2)=\big( R_1(x_1, x_2), R_2(x_1, x_2) \big)$ for $x_1, x_2 \in X$. Moreover, $\overbar{R}$ denotes the inverse of $R$, and $\overbar{R_1}$ and $\overbar{R_2}$ are its components.

\par

A Yang-Baxter set $(X,R)$ is called a {\it birack} if
\begin{enumerate}
\item The map $R_1$ is {\it left-invertible}, that is, for any $x,z \in X$, there is a unique $y \in X$ such that $R_1(x,y)= z$.
\item The map $R_2$ is {\it right-invertible}, that is, for any $y,w \in X$, there is unique $x \in X$ such that $R_2(x,y)=w$.
\end{enumerate}
\par 

A {\it biquandle} is a birack $(X,R)$ satisfying the {\it type I condition}, that is, for a given $a \in X$ there exists a unique $x \in X$ such that $R(x,a)=(x,a)$. Note that the above condition is equivalent to the following: given an element $a \in X$, there exists a unique $x \in X$ such that $R(a,x)=(a,x)$. For more details, see \cite{MR2398735,MR2100870} and \cite[Remark 3.3]{MR2128041}.

\begin{examples}
\begin{enumerate}
\item Let $(X,*)$ be a quandle. Define $R : X \times X \to X \times X$ as $R(x,y)=(y, x*y)$. Then $(X,R)$ is a biquandle.
\item Let $\mathbb{Z}_n$ be a cyclic group of order $n$. Define $R: \mathbb{Z}_n \times \mathbb{Z}_n \to \mathbb{Z}_n \times \mathbb{Z}_n$ as $R(x,y)=(y+1, x-1)$. Then $(\mathbb{Z}_n, R)$ is a biquandle.
\item Let $G$ be a group. Define $R, R': G \times G \to G \times G$ as $R(x,y)=(y^{-1}, yxy)$ and $R'(x,y)=(x^{-1}y^{-1}x, y^2x)$. Then $(G, R)$ and $(G, R')$ are biquandles, called {\it Wada biquandles}.
\item Let $K$ be a commutative ring with unit $1$. Let $\alpha, \beta$ are units in $K$ such that $(1-\alpha) (1-\beta)=0$. Then   $\displaystyle R= \left[ \begin{array}{cc} 1-\alpha & \alpha \\ \beta & 1-\beta \end{array} \right] $ makes the ring $K$ into a biquandle and is known as {\it Alexander biquandle}.
\end{enumerate}

\end{examples}

A {\it precubical set} is a graded set $P=(P_n)_{n \geq 0}$ with {\it boundary  operators} $\partial^k_i: P_n \to P_{n-1}$ $(n > 0, ~k=0,1,~i=1, \ldots, n)$ satisfying the equations $\partial^k_i \circ \partial_j^l = \partial_{j-1}^l \circ \partial_i^k$ for $k,l =0,1,$ and $ i <j$.

There are two approaches for homology theory of the set-theoretic Yang-Baxter sets: one is algebraic \cite{MR2128041}, and the other is graphical \cite{MR3381331}. Both of them give the the same theory \cite{MR3835755}. For our purpose, we are recalling the graphical approach.

Let $(X,R)$ be a Yang-Baxter set. Let $C_n^{\YB}(X)$ be the free abelian group generated by the elements in $X^n$. For each $n \geq 1$, define a homomorphism
\begin{align*}
\partial_n \colon &C_n^{\YB} \to C_{n-1}^{\YB}(X)~\textrm{as}\\
\partial_n = &\sum_{i=1}^{n} (-1)^i \partial_{i,n},
\end{align*}
where $\partial_{i,n}= \partial^l_{i,n} - \partial^r_{i,n}$. Graphically, the maps $\partial_{i,n}^l$, $\partial_{i,n}^r$ and $\partial_{i,n}$ are illustrated in Figure \ref{fig:original_left_map}, Figure \ref{fig:original_right_map} and Figure \ref{fig:original_face_maps}, respectively.

It can be verified that the graded set $(X^n)_{n \geq0 }$ with boundary maps $\partial^{\epsilon}_{i,n}$ where $\epsilon=l,r$, forms a precubical set. Consequently, $(C_*^{\YB}(X), \partial_n)$ constitutes a chain complex. One can now proceed in the standard manner to define homology and cohomology groups.
\begin{figure}[H]
\begin{center}

\includegraphics[height=3in,width=6in,angle=00]{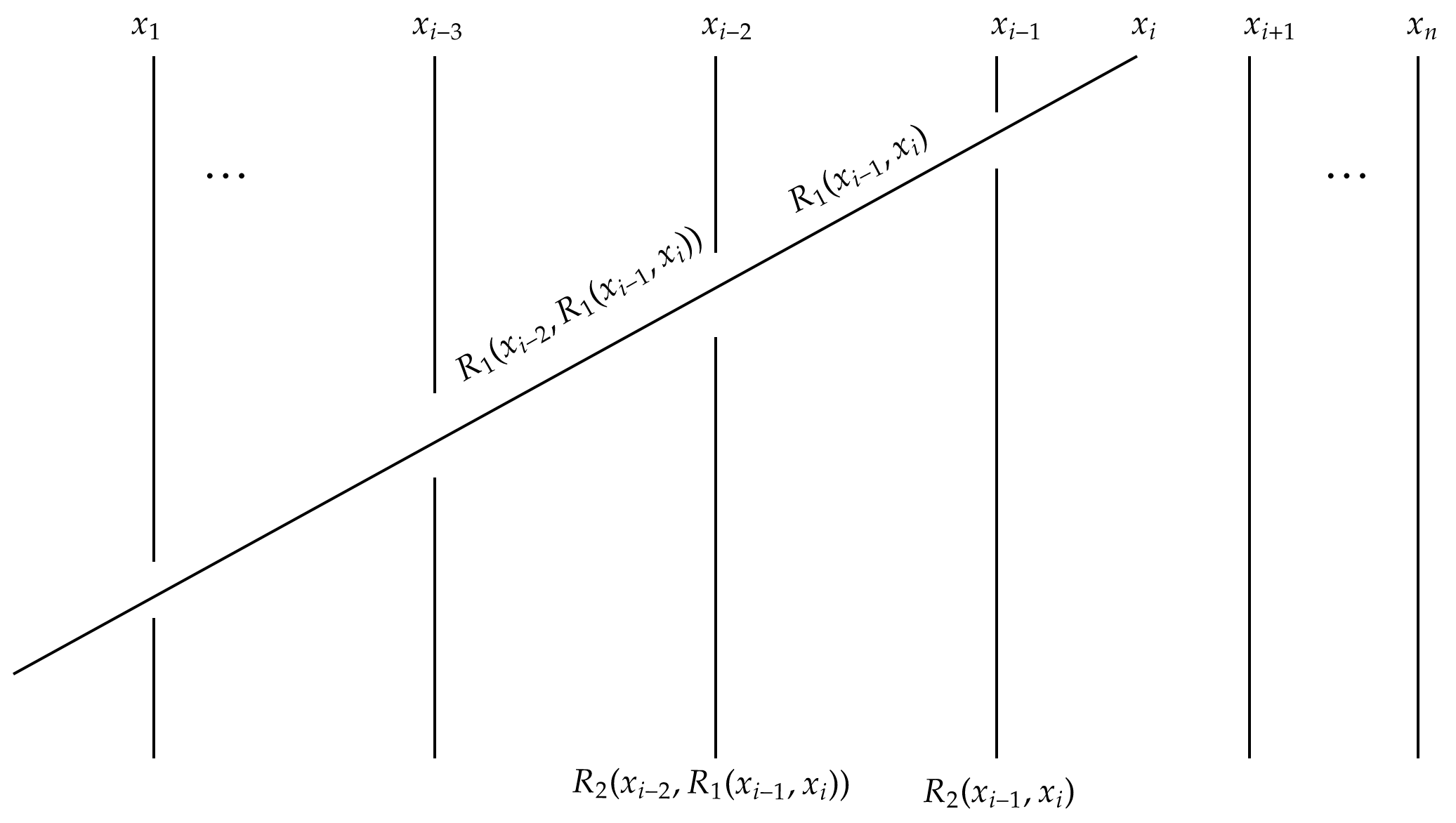}
\end{center}

\caption{The face map $\partial^l_{i,n}.$}
\label{fig:original_left_map}

\end{figure}

\begin{figure}[H]

\begin{center}

\includegraphics[height=3in,width=5in,angle=00]{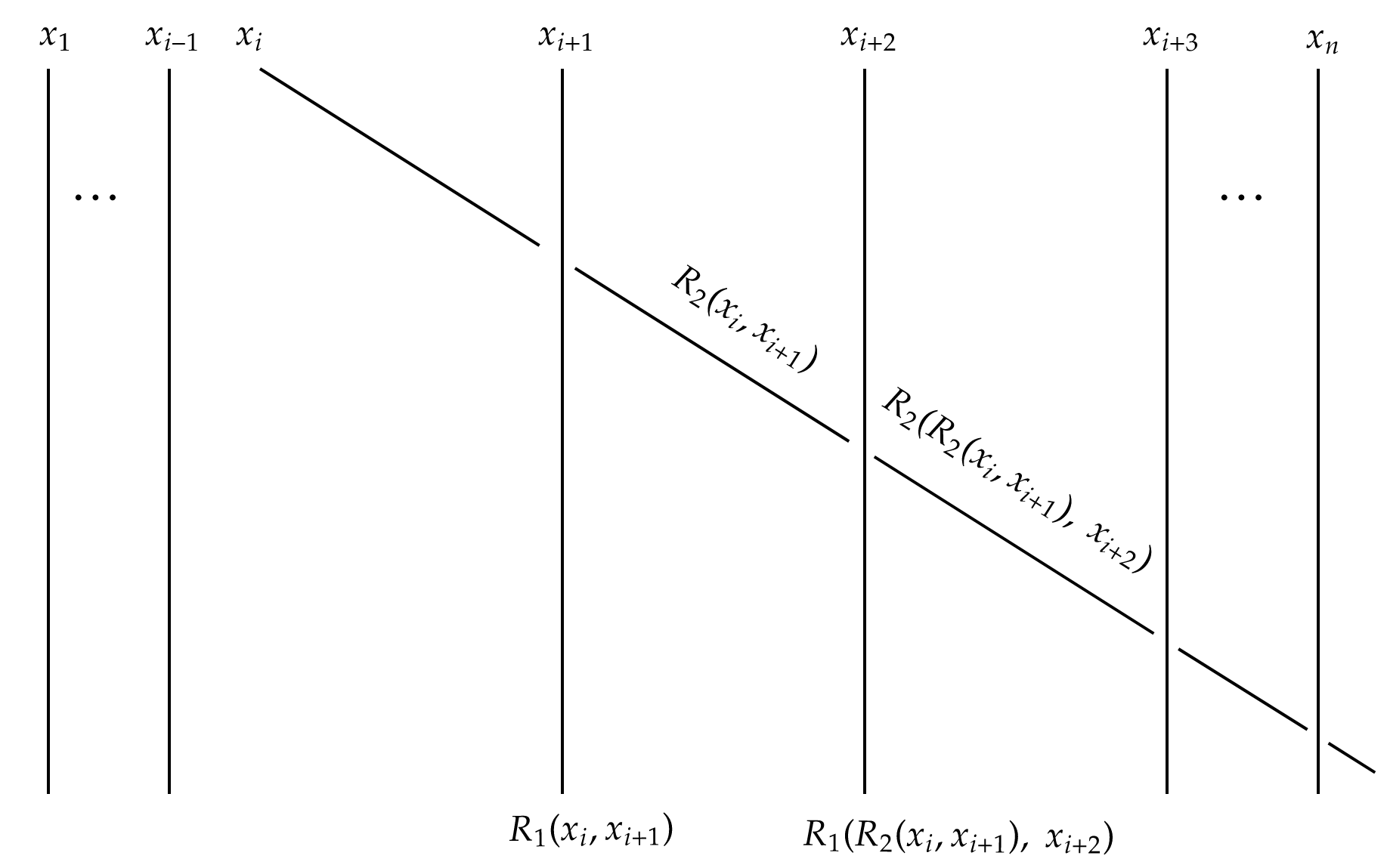}
\end{center}

\caption{The face map $\partial^r_{i,n}.$}
\label{fig:original_right_map}

\end{figure}

\begin{figure}[H]
\begin{center}

\includegraphics[height=2.0in,width=5.0in,angle=00]{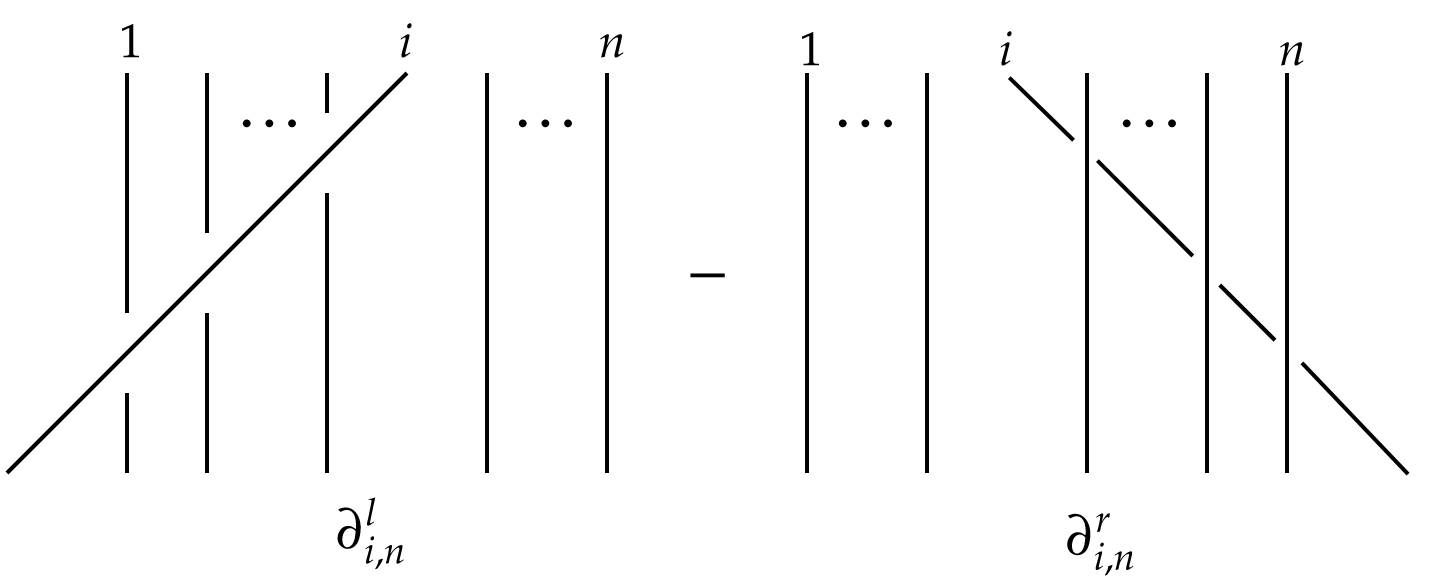}
\end{center}

\caption{The map $\partial_{i,n}.$}
\label{fig:original_face_maps}

\end{figure}


\section{Twisted set-theoretic Yang-Baxter solutions}\label{sec:twisted_yang_baxter_solutions}
In this section, we introduce {\it twisted set-theoretic Yang-Baxter solutions}.

\begin{definition}
A {\it twisted set-theoretic Yang-Baxter solution} is a triplet $(X,f,R)$, where $(X,R)$ is a Yang-Baxter set and $f \colon X \to X$ an automorphism of $(X,R)$. This means that $f $ is a bijective map such that the following diagram commutes:

\begin{center}
\begin{tikzcd}
X \times X \arrow[r, "R"] \arrow[d, "f \times f"'] & X \times X \arrow[d, "f \times f"] \\
X \times X \arrow[r, "R"']                         & X \times X                        
\end{tikzcd}
\end{center}

\end{definition}
For brevity, we refer to $(X,f,R)$ a {\it twisted Yang-Baxter set.}

Note that if $(X,f,R)$ is a twisted Yang-Baxter set, then the following diagrams commute:

\begin{center}
\begin{tikzcd}
X \times X \arrow[r, "R"] \arrow[d, "f^{-1} \times f^{-1}"'] & X \times X \arrow[d, "f^{-1} \times f^{-1}"] &                                                                    &                                              & X \times X \arrow[r, "\bar{R}"] \arrow[d, "f \times f"'] & X \times X \arrow[d, "f \times f"] \\
X \times X \arrow[r, "R"']                                   & X \times X                                   &                                                                    &                                              & X \times X \arrow[r, "\bar{R}"']                         & X \times X                         \\
                                                             &                                              & X \times X \arrow[r, "\bar{R}"] \arrow[d, "f^{-1} \times f^{-1}"'] & X \times X \arrow[d, "f^{-1} \times f^{-1}"] &                                                          &                                    \\
                                                             &                                              & X \times X \arrow[r, "\bar{R}"']                                   & X \times X                                   &                                                          &                                   
\end{tikzcd}
\end{center}

A {\it twisted Yang-Baxter set homomorphism} of $(X,f,R)$ to $(Y,f',R')$ is defined to be a map $\phi: (X,f,R) \to (Y,f', R')$ such that the following diagrams commute:
\begin{center}
\begin{tikzcd}
X\times X \arrow[rr, "R"] \arrow[d, "\phi \times \phi"'] &  & X \times X \arrow[d, "\phi \times \phi"] &  & X \arrow[r, "f"] \arrow[d, "\phi"'] & X \arrow[d, "\phi"] \\
Y \times Y \arrow[rr, "R'"']                             &  &      Y \times Y                          &  & Y \arrow[r, "f'"']                  & Y                  
\end{tikzcd}
\end{center}

A {\it twisted birack} (or {\it twisted biquandle}) $(X,f,R)$ is defined as a twisted Yang-Baxter set $(X,f,R)$, where $(X,R)$ is a birack (or biquandle). In the study of virtual knots (see \cite{MR2191942}), $(X,f,R)$ is called a {\it virtual birack} (or {\it virtual biquandle}).

\par 
Let $(X,f,R)$ be a twisted Yang-Baxter set, where $R=(R_1, R_2)$. Define a map \[\mathcal{T}R : X \times X \to X \times X,\] as
\[
\mathcal{T}R(x,y)= (\mathcal{T}R_1(x,y), \mathcal{T}R_2(x,y)):=(R_1(x,f(y)), R_2(f^{-1}(x),y)).
\]
The map $\mathcal{T}R$ is invertible, with the inverse map $\overbar{\mathcal{T}R}: X \times X \to X \times X $ is 
\begin{align}
\overbar{\mathcal{T}R}(x,y)=\big( \overbar{R_1}(x,f(y)), \overbar{R_2}(f^{-1}(x),y) \big). \label{eq:TR_definition}
\end{align}

Given that $R$ satisfies the set-theoretic Yang-Baxter equation, for all $a,b,c \in X$, Equations \eqref{eq:R_III_1}, \eqref{eq:R_III_2} and \eqref{eq:R_III_3} hold, which we use in the proof of Lemma~\ref{lemma: TR_satisfy_YBE}.
\begin{align}
R_1 \big( R_1(a,b), R_1(R_2(a,b),c) \big)&= R_1 \big(a, R_1(b,c)\big) \label{eq:R_III_1}\\
R_2\big(R_1(a,b), R_1(R_2(a,b),c)\big)&=R_1\big(R_2(a,R_1(b,c)), R_2(b,c)\big) \label{eq:R_III_2}\\
R_2\big(R_2(a,b),c\big) &= R_2\big(R_2(a, R_1(b,c)), R_2(b,c)\big) \label{eq:R_III_3}
\end{align}
\begin{lemma}\label{lemma: TR_satisfy_YBE}
The map $\mathcal{T}R: X \times X \to X \times X$ satisfies the set-theoretic Yang-Baxter equation, that is,
\[
(\mathcal{T}R \times \id_X ) (\id_X \times \mathcal{T}R) (\mathcal{T}R \times \id_X) = (\id_X \times \mathcal{T}R) 
 (\mathcal{T}R \times \id_X)
(\id_X \times \mathcal{T}R)
\]
\end{lemma}
\begin{proof}L.H.S=
\begin{align*}
(\mathcal{T}R \times \id_X ) (\id_X \times \mathcal{T}R) (\mathcal{T}R \times \id_X)~~(x,y,z)=&(\mathcal{T}R \times \id_X) (\id_X \times \mathcal{T}R)~~ \big( R_1(x,f(y)), R_2(f^{-1}(x),y), z \big)\\
=&(\mathcal{T}R \times \id_X ) ~~ \big( R_1(x,f(y)), R_1(  R_2 (f^{-1}(x), y), f(z) ),\\& R_2 ( R_2 (f^{-2}(x), f^{-1}(y), z) \big)\\
=&\big( R_1 \big( R_1(x,f(y), R_1(R_2(x,f(y)), f^2(z) ) \big), \\&R_2 \big( R_1 (f^{-1}(x), y), R_1 (R_2(f^{-1}(x),y), f(z)\big),\\& R_2 ( R_2 (f^{-2}(x), f^{-1}(y), z) \big)
\end{align*}

R.H.S=
\begin{align*}
(\id_X \times \mathcal{T}R)
 (\mathcal{T}R \times \id) 
(\id_X \times \mathcal{T}R)~~(x,y,z)=&(\id_X \times \mathcal{T}R) 
 (\mathcal{T}R \times \id_X)
 ~~\big( x, R_1(y,f(z)), R_2 (f^{-1}(y),z) \big)\\
=&(\id_X \times \mathcal{T}R) ~~R_1\big( x, R_1(f(y), f^2(z)), R_2(f^{-1}(x), R_1(y, f(z))),\\& R_2(f^{-1}(y),z) \big)\\
=& \big( R_1(x, R_1(f(y), f^2(z))), R_1(R_2(f^{-1}(x), R_1(y, f(z))), R_2(y, f(z)) \big),\\
&R_2(R_2(f^{-2}(x), R_1(f^{-1}(y),z)), R_2(f^{-1}(y),z)) \big)
\end{align*}

Now L.H.S=R.H.S, if and only if, the following equations hold for all $x,y,z \in X$:
\begin{align}
R_1 \big( R_1(x,f(y), R_1(R_2(x,f(y)), f^2(z) ) \big)= R_1\big(x, R_1(f(y), f^2(z))\big) \label{eq:TR_III_1}\\
R_2 \big( R_1 (f^{-1}(x), y), R_1 (R_2(f^{-1}(x),y), f(z)\big)= R_1\big(R_2(f^{-1}(x), R_1(y, f(z))\big), R_2(y,f(z))\big) \label{eq:TR_III_2}\\
R_2 \big( R_2 (f^{-2}(x), f^{-1}(y), z\big) =R_2\big(R_2(f^{-2}(x), R_1(f^{-1}(y),z)), R_2(f^{-1}(y),z)\big) \label{eq:TR_III_3}
\end{align}

Equation \ref{eq:TR_III_1} is derived from Equation \ref{eq:R_III_1} by substituting $a$ with $x$, $b$ with $f(y)$ and $c$ with $f^2(z)$. Equation \ref{eq:TR_III_2} is obtained from Equation \ref{eq:R_III_2} by substituting $a$ with $f^{-1}(x)$, $b$ with $y$ and $c$ with $f(z)$. Equation \ref{eq:TR_III_3} is derived from Equation \ref{eq:R_III_3} by substituting $a$ with $f^{-2}(x)$, $b$ with $f^{-1}(y)$ and $c$ with $z$.
\end{proof}

For a given integer $t$, we define the operator $\mathcal{T}^tR: X \times X \to X \times X$ as follows:
\[
\mathcal{T}^tR(x,y)= (\mathcal{T}^tR_1(x,y), \mathcal{T}^tR_2(x,y)):=(R_1(x,f^t(y)), R_2(f^{-t}(x),y)).
\]
Lemma \ref{lemma: TR_satisfy_YBE} establishes that for every $t \in \mathbb{Z}$, the triplet $(X, f, \mathcal{T}^tR)$ constitutes a twisted Yang-Baxter set. 

Furthermore, we state the following results without proof as they are straightforward.

\begin{proposition}
Let $(X,f,R)$ and $(X',f', R')$ be twisted Yang-Baxter sets. If $(X,f,R) \cong (X',f', R')$, then $(X, f, \mathcal{T}^tR) \cong (X', f', \mathcal{T}^tR')$ for all $t \in \mathbb{Z}$.
\end{proposition}

\begin{corollary}\label{corollary:twisted_Yang-Baxter_set_is_Yang-Baxter_set}
If $(X,f,R)$ is a twisted Yang-Baxter set, then the following properties hold:
\begin{enumerate}
\item $(X,\mathcal{T}R)$ is a Yang-Baxter set.
\item If $(X,f,R)$ is a twisted birack, then $(X,\mathcal{T}R)$ is a birack.
\item If $(X,f,R)$ is a twisted biquandle, then $(X,\mathcal{T}R)$ is a biquandle.
\end{enumerate}
\end{corollary}

\begin{corollary}
Let $(X,f, R)$ be a twisted birack. Then $(X,f,R) \cong (X, f , \mathcal{T}R)$ if and only if $f\colon X \to X$ is the identity map.
\end{corollary}

\begin{remark}
A {\it twisted rack} (or {\it twisted quandle}) $(X, f, *)$ is a rack (or quandle) $(X, *) $ equipped with an automorphism $f: X \to X$. Note that $(X,f, *)$ can also be represented as $(X,f ,R)$, where $R(x,y)=(y, x*y)$ for all $x, y \in X$. However, under the operator $\mathcal{T}R$, $(X, \mathcal{T}R)$ is not a rack (or quandle) unless $f$ is the identity map on $X$.  
\end{remark}


\section{Twisted Yang-Baxter (co)homology theory}\label{sec:homology_theory_for_twisted_yang_baxter_set}
In this section we introduce a cohomology theory for twisted-Yang Baxter sets, which we term the {\it twisted set-theoretic Yang-Baxter cohomology theory}.

Let $(X,f,R)$ be a twisted Yang-Baxter set, and $(t,m_1, m_2) \in \mathbb{Z}^3$. Consider the ring $\Lambda=\mathbb{Z}[T,T^{-1}]$ of Laurent polynomials over the integers. For each integer $n>0$, $C_n^{\TYB}(X)$ is a $\Lambda$-module over the generating set $X^n$, where
\[T.(x_1, x_2, \ldots, x_n)=(f(x_1), f(x_2), \ldots, f(x_n)),
\]
for all $(x_1, x_2, \ldots, x_n)\in X^n$.
For each $n$, we define the $n$-boundary homomorphism
\[
\partial_{n}^{(t,m_1, m_2)}: C_n^{\TYB} (X) \to C_{n-1}^{\TYB}(X)
\]
as
\[
\partial_{n}^{(t,m_1, m_2)} = \sum_{i=1}^{n} (-1)^i \partial^{(t,m_1, m_2)}_{i,n},
\]
where
\[
\partial_{i,n}^{(t,m_1, m_2)}= \partial_{i,n}^{l,(t,m_1)} - \partial_{i,n}^{r, (t,m_2)}.
\]
The face maps $\partial_{i,n}^{l,(t,m_1)}$ and $\partial_{i,n}^{r,(t,m_2)}$ are illustrated graphically in Figures \ref{fig:twisted_left_map} and \ref{fig:twisted_right_map}.

\begin{figure}[H]
\begin{center}

\includegraphics[height=3in,width=6in,angle=00]{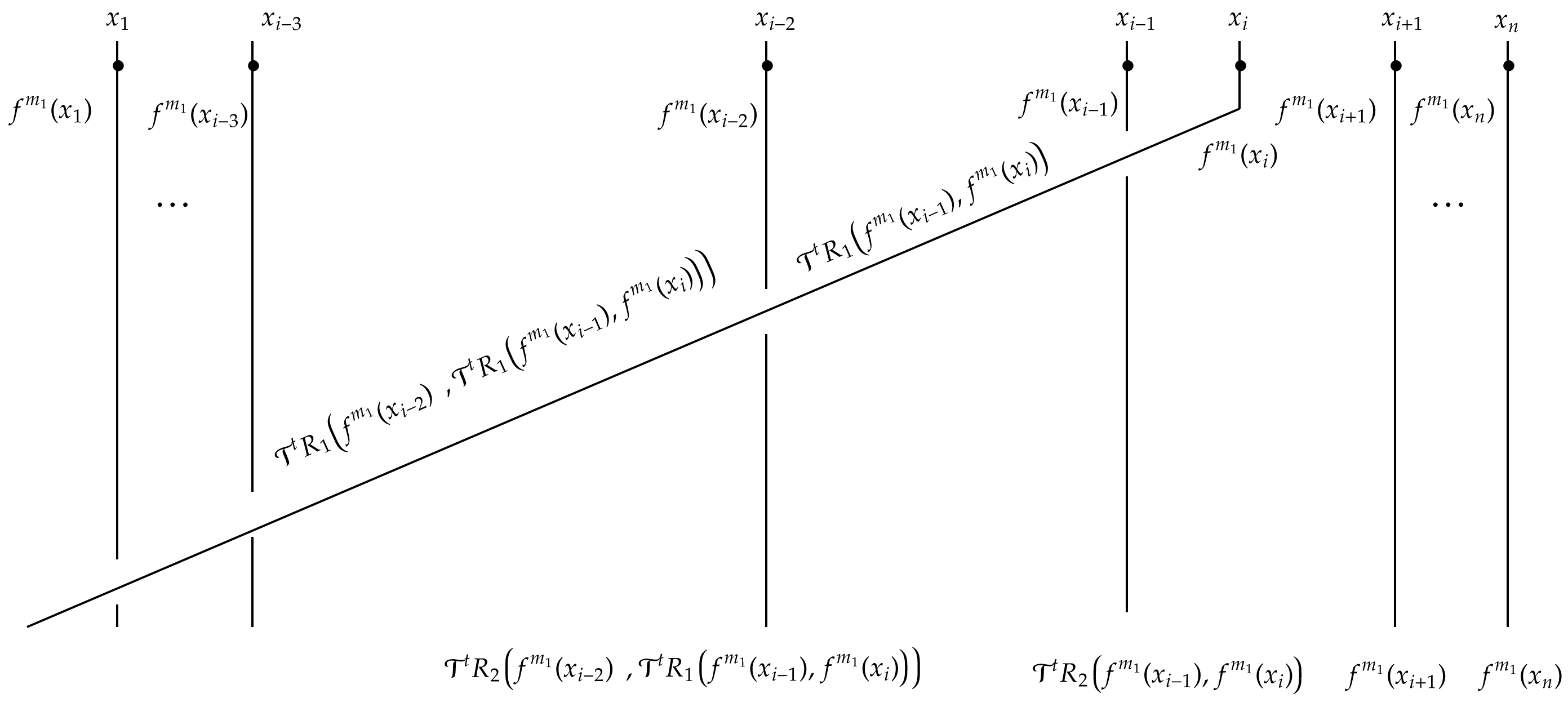}
\end{center}
\caption{A face map $\partial^{l,(t,m_1)}_{i,n}.$}
\label{fig:twisted_left_map}
\end{figure}

\begin{figure}[H]
\begin{center}

\includegraphics[height=3.5in,width=6in,angle=00]{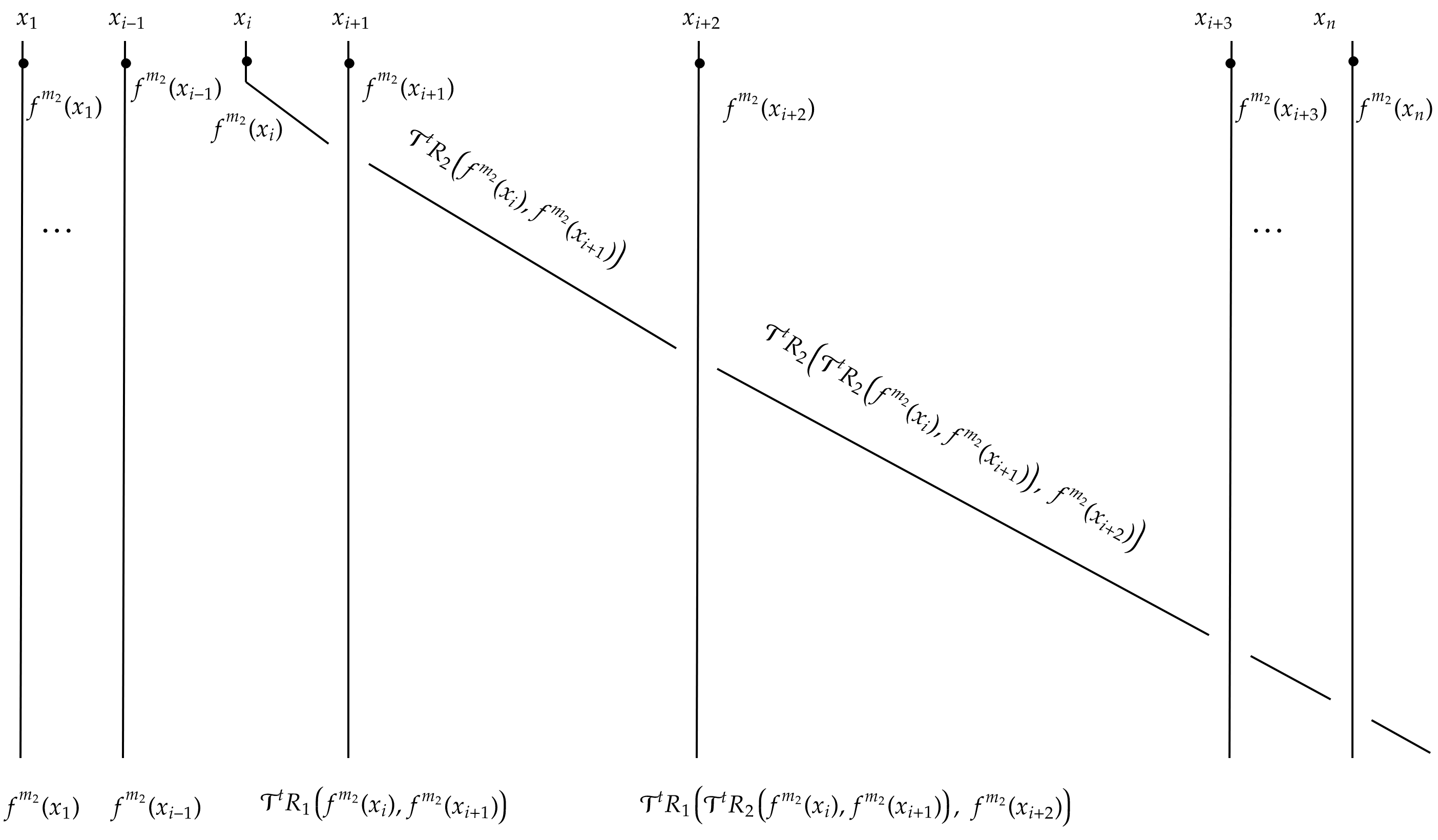}
\end{center}
\caption{A face map $\partial^{r,(t,m_2)}_{i,n}.$}
\label{fig:twisted_right_map}
\end{figure}

We define $C_0^{\TYB}(X)=\{0\}$. It is easy to check that the graded set $(X^n)_{n \geq 0}$ along with the boundary maps $\partial_{i,n}^{l,(t,m_1)}$ and $\partial_{i,n}^{r, (t,m_2)}$, forms a precubical set. This implies that $
C_{*}^{\TYB, (t,m_1, m_2)}(X):=\big( C_{n}^{\TYB}(X), \partial_{n}^{(t,m_1, m_2)} \big)
$ is a chain complex.

For a given $\Lambda$-module $M$, consider the chain and cochain complexes
\begin{align*}
&C_{*}^{\TYB,(t,m_1, m_2)}(X;M) = C_{*}^{\TYB, (t,m_1, m_2)}(X) \otimes_{\Lambda} M, && \partial^{(t,m_1, m_2)}= \partial^{(t,m_1, m_2)} \otimes_{\Lambda} \id_M;\\
&C^{*}_{\TYB, (t,m_1, m_2)}(X;M)= \Hom_{\Lambda}\big(C_{*}^{\TYB, (t,m_1, m_2)}(X),M\big) && \delta_{(t,m_1, m_2)}=\Hom(\partial^{(t,m_1, m_2)}, \id_M).
\end{align*}
The $n$-th homology and cohomology groups of these complexes are called {\it $t$-$(m_1, m_2)$- twisted Yang-Baxter homology group} and {\it cohomology group}, and are denoted by $H_{*}^{\TYB, (t,m_1, m_2)}(X;M)$ and $H^{*}_{\TYB, (t,m_1, m_2)}(X;M)$, respectively. The (co)cycles  are referred to as {\it $t$-$(m_1, m_2)$-Yang-Baxter $($co$)$cycles}.

By taking $f$ to be the identity map on $X$, we recover the cohomology theory of Yang-Baxter sets introduced in \cite{MR2128041, MR3381331}.

\begin{definition}
For a given $(t, m_1, m_2) \in \mathbb{Z}^3$, the (co)homology theory defined above is termed the {\it $t$-$(m_1, m_2)$-$($co$)$homology theory of the set-theoretic twisted Yang-Baxter set} $(X,f,R)$.
\end{definition}

\begin{proposition}\label{prop:can_take_m1_0}
Let $(X,f,R)$ be a twisted Yang-Baxter set, $M$ a $\Lambda$-module, and $t \in \mathbb{Z}$. Suppose $(m_1, m_2), (n_1,n_2) \in \mathbb{Z}^2$, such that $n_1=m_1 + k$ and $n_2=m_2 +k$ for some $k \in \mathbb{Z}$. Then, for each $n \in \mathbb{Z}$,  $H_n^{\TYB, (t,m_1, m_2) }(X;M)\cong H_n^{\TYB, (t,n_1, n_2) }(X;M)$ and $H^n_{\TYB, (t,m_1, m_2)}(X;M) \cong H^n_{\TYB, (t,n_1, n_2) }(X;M)$, where these isomorphisms are $\Lambda$-module isomorphisms.
\end{proposition}
\begin{proof}
Observe that for each $n \in \mathbb{Z}^+$, \[\partial_n^{(t,n_1, n_2)} \otimes_{\Lambda} \id_M = (T^k\otimes_{\Lambda} \id_M) \circ (\partial_n^{(t,m_1, m_2)} \otimes_{\Lambda} \id_M),\] and
\[
\delta^n_{(t,n_1, n_2)}=T^k \circ \delta^n_{(t,m_1, m_2)}.
\]
\end{proof}

\begin{remark}
Moving forward, our primary focus will be on studying the cohomology theory of $(X,f,R)$, with $t=0$, unless specifically stated otherwise. Nevertheless, all subsequent results presented in the paper remain valid for non-zero values of $t$. Furthermore, for $t=0$ and by Proposition \ref{prop:can_take_m1_0} considering $m_1=0$, we will simplify the notations. For instance,  $C_*^{\TYB, (0, 0, m)}(X;M)$ will be represented as $C_*^{\TYB, (m)}(X;M)$ and $C^{*}_{\TYB, (0,0, m)}(X, M)$ as $C^{*}_{\TYB, (m)}(X, M)$, and the same simplification will be applied to the homology and cohomology groups.
\end{remark}

Let $(X,f,R)$ be a twisted biquandle. Consider a $\Lambda$-submodule, denoted as $C_n^{\TD}(X)$, of $C_n^{\TYB}(X)$ defined as
\[
C_n^{\TD}(X)= \textrm{ span}\{ (x_1, \ldots, x_n) \in C_n^{\TYB}(X)~|~ R(x_i, x_{i+1}) =(x_i, x_{i+1})
\textrm{ for some }i=1,\ldots, n-1\}
\]
if $n \geq 2$; otherwise, $C_n^{\TD}(X)=0$.

The following result is easy to prove, and we leave it for the readers.

\begin{proposition}\label{prop:sub_chain_complex}
Let $(X,f,R)$ be a twisted biquandle. Then $ \partial_n^{(m)}(C_n^{\TD}(X)) \subseteq C_{n-1}^{\TD}(X)$ and  $(C_n^{\TD}(X), \partial_n^{(m)}) $ is a sub-chain complex of $(C_n^{\TYB}(X), \partial_n^{(m)})$.
\end{proposition}

By Proposition \ref{prop:sub_chain_complex}, for a given twisted biquandle $(X,f,R)$, we have a quotient chain complex $C_*^{\TBQ, (m)}(X)= (C_n^{\TBQ}, \partial_n^{(m)})$, where $C_n^{\TBQ}(X)= C_n^{\TYB}(X)/C_n^{\TD}(X)$, and $\partial_n^{(m)}$ is the induced homomorphism. For a $\Lambda$-module $M$, define the chain and cochain complexes $C_*^{\TBQ, (m)}(X;M)=\big(C_n^{\TBQ}(X;M), \partial_n^{(m)}\big)$ and $C^*_{\TBQ, (m)}(X;M)=(C^n_{\TBQ}(X;M), \delta_{(m)}^{n})$, where 
\begin{align*}
C_n^{\TBQ}(X;M) &= C_n^{\TBQ}(X) \otimes_{\Lambda} M, && \partial_n^{(m)}= \partial_n^{(m)} \otimes \id_M,\\
C^n_{\TBQ}(X;M) &= \Hom_{\Lambda}(C_n^{\TBQ}(X), M), && \delta^n_{(m)} = \Hom(\partial_n^{(m)}, \id_M)
\end{align*}

For $W=\TYB, \TBQ$ (denoting a twisted Yang Baxter, and a twisted biquandle, respectively) and a given $m$, the $n$-th group of cycles and boundaries are denoted by $\ker(\partial^{(m)})= Z_n^{W,(m)}(X;M) \subseteq C_n^W(X;M)$ and $\im(\partial^{(m)})=B_n^{W,(m)}(X;M) \subseteq C_n^{W}(X;M),$ respectively. The $n$th group of cocycles and coboundaries  are denoted respectively by $\ker(\delta_{(m)})=Z^n_{W,(m)}(X;M) \subseteq C^n_W(X;M)$ and $\im(\delta_{(m)})=B^n_{W,(m)} \subseteq C^n_W(X;M)$. Thus the (co)homology groups are given as quotients:

\begin{align*}
H_n^{W,(m)}(X;M) &= Z_n^{W, (m)}(X;M)/B_n^{W, (m)}(X;M),\\
H^n_{W, (m)}(X;M) &= Z^n_{W,(m)}(X;M)/ B^n_{W, (m)}(X;M).
\end{align*}

Let $(X,f,R)$ be a twisted Yang-Baxter set and $M$ a $\Lambda$-module.
\begin{itemize}
\item Then the $1$-cocycle condition for $\eta \in Z^1_{\TBQ,(m)}(X;M)$ is
\[
-\eta(x_2) + \eta(R_1(f^m(x_1), f^m(x_2))) + \eta(R_2(x_1, x_2)) - \eta(f^m(x_1)) =0
\]
which can also be expressed as
\[
-\eta(x_2) + T^m(\eta(R_1(x_1, x_2)) + \eta(R_2(x_1, x_2)) - T^m(\eta(x_1)) =0
\]
\item  The $2$-cocycle condition for $\phi \in Z^2_{\TBQ, (m)}(X;M)$ is 
\begin{align*}
&\phi( R_1(f^{m}(x_1), f^{m}(x_2) ), R_1( R_2( f^{m}(x_1), f^{m}(x_2) ) , f^{m}(x_3) ) ) +\\ &\phi(R_2(x_1, x_2 ), x_3 )  + \phi(f^{m}(x_1), f^{m} (x_2) )
=\\
& \phi(x_2, x_3 ) + \phi( f^{m}(x_1), R_1( f^{m} (x_2), f^{m} (x_3) ) ) +\\
& \phi( R_2(x_1, R_1(x_2, x_3 )) , R_2 ( x_2 , x_3 ) )
\end{align*}
which can also be expressed as
\begin{equation}\label{eq:2-cocycle}
\begin{split}
&T^m(\phi( R_1(x_1, x_2 ), R_1( R_2( x_1, x_2 ) , x_3 ) ) +\\ &\phi(R_2(x_1, x_2 ), x_3 )  + T^m\big(\phi(x_1, x_2) \big)
=\\
& \phi(x_2, x_3 ) + T^m\big(\phi( x_1, R_1(x_2, x_3 ) )\big) +\\
& \phi( R_2(x_1, R_1(x_2, x_3) ) , R_2 ( x_2 , x_3 ) )
\end{split}
\end{equation}
If $(X,f,R)$ is a twisted biquandle, then a $2$-cocycle $\phi \in Z_{\TBQ, (m)}^2(X;M)$ satisfy the Equation \ref{eq:2-cocycle} and $\phi(x,y)=0$ for all $x, y \in X$ such that $R(x,y)=(x,y)$.
\end{itemize}

\section{Calculations of Second (Co)homology}\label{sec:Calcualtions_of_homology}
In this section, we compute the second (co)homology groups of some twisted biquandles.

\begin{example}
Let $m=1$. Consider the twisted quandle $(T_2, I, *)$ where $(T_2, *)$ is a trivial quandle with two elements $\{0,1\}$, and $I$ is the identity permutation on $\{0, 1\}$. In this case, every (co)chain is a (co)cycle and zero is the only (co)boundary (see Lemma \cite[Lemma 6.5]{
MR1990571}).
\end{example}

\begin{example}\label{example:2}
Let $n \in \mathbb{Z}_{\geq 0}$ and $m=1$. Consider the twisted quandle $X=(T_2, f, *)$ where $(T_2, *)$ is a trivial quandle with two elements $\{0,1\}$, and $f$ is the permutation $(0~1)$. Then $C_1(X; \Lambda_n)$ is a module over $\Lambda_n$ generated by $\langle (0) \rangle$, $C_2(X; \Lambda_n)$ is a module over $\Lambda_n$ generated by $\langle (0,1) \rangle$ and $C_3(X; \Lambda_n)$ is a module over $\Lambda_n$ generated by $\langle (0,1,0) \rangle$. We have
$\partial_3^{(1)}((0,1,0))=0$, thus $B_2^{\TBQ, (1)}(X;\Lambda_n)=\{0\}$, and 
$\partial_2^{(1)}((0,1))=2((1)-(0)).$
When $n$ is odd, then $Z_2^{\TBQ, (1)}(X;\Lambda_n)$ is a module over $\Lambda_n$ generated by $\{ (1,0)+(0,1)\}$, and is isomorphic to the group $\mathbb{Z}_n$. Thus $H_2^{TBQ, (1)}(X;\Lambda_n)\cong \mathbb{Z}_n$ as a group. For $n$ even, $Z_2^{\TBQ, (1)}(X;\Lambda_n)$ is a module over $\Lambda_n$ generated by $\{ (1,0)+(0,1), \frac{n}{2}(0,1) \}$. Thus $H_2^{\TBQ, (1)}(X;\Lambda_n)\cong \mathbb{Z}_n \oplus \mathbb{Z}_n$ as a group.
\end{example}

\begin{example}
Take $X=(T_2,f, *)$ as stated in Example \ref{example:2}. Consider the cochain complex $C^*(X; \Lambda_n/\langle 1-T^2\rangle)$. Let $g \in C^2(X, \Lambda_n/\langle 1-T^2\rangle)$. Then $\delta^3_{(1)} (g)(0,1,0)=0$, thus the kernel of $\delta^3_{(1)}$ is $C^2(X; \Lambda_n/\langle 1-T^2\rangle)$. Now for $g \in C^1(X; \Lambda_n/\langle 1-T^2\rangle)$, $\delta^2_{(1)}( g)((0,1))=2g((0))-2g((1))=(2-2T)g((0)).$ One can check that $C^1(X; \Lambda_n/\langle 1-T^2\rangle)= \{\phi_{(a,b)},\textrm{ where } \phi_{(a,b)}((0))=a + b T,  \textrm{ and }a, b \in \mathbb{Z}_n\}\cong \mathbb{Z}_n \oplus T\mathbb{Z}_n$ and $C^2(X; \Lambda_n/\langle 1-T^2\rangle)= \{ \psi_{(a,b)}, \textrm{ where } \psi_{(a,b)}((0,1))=a + b T, \textrm{ and }a, b \in \mathbb{Z}_n\}\cong \mathbb{Z}_n\oplus T\mathbb{Z}_n.$ If $n=2$, then $H^2_{\TBQ,(1)}(X; \Lambda_n/\langle 1- T^2\rangle \cong \mathbb{Z}_2 \oplus\mathbb{Z}_2$. Now we consider the case when $n$ is odd prime. Then $\delta^2_{(1)}(C^1(X; \Lambda_n/\langle 1-T^2 \rangle)=2(1-T)(\mathbb{Z}_n \oplus T\mathbb{Z}_n)=2(1-T)\mathbb{Z}_n$. Thus $$H^2_{\TBQ, (1)}(X; \Lambda_n/\langle1-T^2\rangle\cong\frac{\mathbb{Z}_n\oplus T\mathbb{Z}_n}{2(1-T) \mathbb{Z}_n} \cong \frac{\mathbb{Z}_n\oplus T\mathbb{Z}_n}{(1-T) \mathbb{Z}_n}\cong \mathbb{Z}_n~ \textrm{ as a group}.$$
\end{example}

\begin{example}
Consider the twisted quandle $X=(R_3, I, *)$, where $(R_3, *)$ is the dihedral quandle on three elements $\{0,1,2\}$, and $I$ is the identity map on $R_3$. Then $H_2^{\TBQ, (1)}=\{0\}$ (see \cite{MR3042590}). 
\end{example}

\begin{example}
Consider the twisted quandle $X=(R_3, f, *)$, where $(R_3, *)$ is the dihedral quandle on three elements $\{0,1,2\}$, and $f(x)=x+1$ which is the permutation $(0~1~2)$. Consider the chain complex $C_*^{\TBQ, (1)}(X; \Lambda_3)$. Then

\begin{align*}
\partial_2^{\TBQ, (1)}((0,1))&=(1)-(2)\\
\partial_2^{\TBQ, (1)}((0,2))&=(0)-(2)\\
\partial_2^{\TBQ, (1)}((1,2))&=(2)-(0)\\
\partial_2^{\TBQ, (1)}((1,0))&=(1)-(0)\\
\partial_2^{\TBQ, (1)}((2,0))&=(0)-(1)\\
\partial_2^{\TBQ, (1)}((2,1))&=(2)-(1).
\end{align*}
It is evident that $Z_2^{\TBQ, (1)}$ is generated by $\{(0,1)+(1,2)+(2,0), (0,1)+(2,1) \}$ over $\Lambda_3$. Noting that $(0,1)+(1,2)+(2,0)=\partial_3^{\TBQ,(1)}((0,1,0))+ \partial_3^{\TBQ,(1)}((2,0,2))-\partial_3^{\TBQ,(1)}((2,0,1))$ and $(0,1)+(2,1)=\partial_3^{\TBQ, (1)}((1,2,0))$, we get $H_2^{\TBQ, (1)}(X; \Lambda_3)=0$.
\end{example}

\begin{example}
Consider the twisted quandle $X=(R_3, f, *)$, where $(R_3, *)$ is the dihedral quandle on three elements $\{0,1,2\}$, and $f$ is the permutation $(1~2)$. Consider the chain complex $C_*^{\TBQ, (1)}(X, \Lambda_3)$. Then
\[\partial_2^{(1)}((0,1))=\partial_2^{(1)}((0,2))=-\partial^{(1)}_2(1,2)=-\partial^{(1)}_2((2,1))=-(0)-(1)-(2),\]
and thus $Z_2^{\TBQ, (1)}\cong \mathbb{Z}_3^5$.
Furthermore, by solving equations derived from the images of the generator under $\partial_3$, one can verify that $B_2^{\TBQ, (1)}\cong \mathbb{Z}_3^3$. Thus $H_2^{\TBQ, (1)} \cong \mathbb{Z}_3^2$.

\end{example}

\section{Extension Theory}\label{sec:extensions}
In this section, we study extensions cocycle extensions of twisted Yang-Baxter sets.

Let $(X,f,R)$ be a twisted Yang-Baxter set, $M$ be a $\Lambda$-module and $\phi_1$ and $ \phi_2$ be maps from $X\times X$ to $M$, such that $\phi_i \circ (T \times T)=(T \times T) \circ \phi_i$ for $i=1,2$.  We then have the following  result.
	
\begin{proposition} \label{extprop}
Let $V=M \times X$, $m_1, m_2 \in \mathbb{Z}$ and $S: V \times V \rightarrow  V \times V $ be defined by
$$S( (a, x), (b, y) ) = ( ~(b + T^{m_1}\phi_1(x, y) , R_1(x, y)), (a+ T^{m_2}\phi_2(x, y),  R_2(x, y))~)$$
for all $(a, x), (b,y) \in V$. Consider a map $\phi: X \times X \to M$ defined as
\[		
\phi(x,y):= T^{m_1}\phi_1(x, y) + T^{m_2}\phi_2(x,y),~\textrm{ for all } x, y \in X.
\]
If  $(V, S)$ is a Yang-Baxter set, then $\phi \in  Z^2_{\rm TYB, (0)}(X;M)$.
\end{proposition}

\begin{proof}
We compute
\begin{align*}		
& (S \times \id_V )( \id_V \times S)(S \times \id_V)(( a, x), (b, y), (c, z) )= \\
& (~(   c + T^{m_1}\phi_1( R_2( x,y) , z) +	T^{m_1}\phi_1(R_1(x, y ), R_1( R_2( x, y), z)) ,R_1( R_1(x, y), R_1(R_2( x, y), z) )),\\
&  (b + T^{m_1}\phi_1(x, y) + T^{m_2}\phi_2( R_1(x,y), R_1( R_2( x, y), z) ) ,R_2( R_1(x, y), R_1(R_2( x, y), z))), \\
& (a +  T^{m_2}\phi_2(x, y) + T^{m_2}\phi_2( R_2( x, y) , z), R_2(R_2(x, y) , z))~)
\end{align*}
and on the other hand,
\begin{align*}	
&(\id_V \times S)(S \times \id_V)( \id_V \times S)(( a, x), (b, y), (c, z) )=\\& (~ ( c + T^{m_1}\phi_1( y, z) + T^{m_1}\phi_1( x, R_1(y, z) ) , R_1(x, R_1(y, z))), \\
& ( b +  T^{m_2}\phi_2( y, z) + T^{m_1}\phi_1(R_2(x, R_1(y, z)),   R_2(y, z) ), R_1(R_2(x, R_1(y, z)),  R_2(y, z))), \\
& ( a + T^{m_2}\phi_2(x, R_1(y, z) )+T^{m_2}\phi_2(R_2(x, R_1(y, z)),   R_2(y, z)) , R_2(R_2(x, R_1(y, z)),  R_2(y, z))) ~) .
\end{align*}
Since $(V,S)$ is a Yang-Baxter set, we obtain the following three equations from each factor containing $c, b, a$ respectively.
\begin{align*}	
T^{m_1}\phi_1( R_2( x,y) , z) + T^{m_1}\phi_1(R_1(x,  y ), R_1( R_2( x, y), z) )&= T^{m_1}\phi_1( y, z) + T^{m_1}\phi_1( x, R_1(y, z) ) \\	
T^{m_1}\phi_1(x, y) + T^{m_2}\phi_2( R_1(x,y),	R_1( R_2( x, y), z) )&=   T^{m_2}\phi_2( y, z) +  T^{m_1} \phi_1(R_2(x, R_1(y, z)),   R_2(y, z) ) \\	
T^{m_2} \phi_2(x, y) + T^{m_2} \phi_2( R_2( x, y) , z)& =  T^{m_2} \phi_2(x, R_1(y,z) ) +   T^{m_2} \phi_2(R_2(x, R_1(y, z)),   R_2(y, z)).
\end{align*}
By adding the equalities above, we obtain, that the map $\phi=\phi_1+ \phi_2$ is in $Z^2_{\TYB, (0)}(X;M)$.
\end{proof}	
	
\begin{remark}
Notice that $\phi_1$ and $\phi_2$ are not $2$-cocycles.
\end{remark}

\begin{proposition}
Let $V=M \times X$, $m_1, m_2 \in \mathbb{Z}$,  and $S: V \times V \rightarrow  V \times V $ be defined by 
\[S( (a, x), (b, y) ) = ( (b + T^{m_1}\phi(x, y) , R_1(x, y) ), (a+ T^{m_2}\phi(x, y),  R_2(x, y) )\]
for all $(a, x), (b,y) \in V$ and where  $\phi: X \times X \rightarrow M$ and $\phi \circ (T \times T)=(T \times T) \circ \phi$.  Define	
\[		
\psi(x,y):= T^{m_1}\phi(x, y) + T^{m_2}\phi(x, y).
\]
If  $(V, S)$ is a Yang-Baxter set, then $\psi \in  Z^2_{\TYB, (1,m_1,m_2)}(X;M)$.	
\end{proposition}

\begin{proof}
We compute
\begin{align*}		
&(S \times 1 )( 1 \times S)(S \times 1)(	( a, x), (b, y), (c, z) ) =\\& ( ~(   c + T^{m_1}\phi( R_2( x,y) , z) +	
		T^{m_1}\phi(R_1(x,  y ), R_1( R_2( x, y), z) ) , R_1( R_1(x, y), R_1(R_2( x, y), z) ) \; ),\\
& (b + T^{m_1}\phi(x, y) + T^{m_2}\phi( R_1(x,y), R_1( R_2( x, y), z) ) , R_2( R_1(x, y), R_1(R_2( x, y), z) )), \\
& ( a +  T^{m_2}\phi(x, y) + T^{m_2}\phi( R_2( x, y) , z), R_2(R_2(x, y) , z)) ~)
\end{align*}
and on the other hand,
\begin{align*}	
&( 1 \times S)(S \times 1)( 1 \times S)(( a, x), (b, y), (c, z) )  = \\
& ( ( c + T^{m_1}\phi( y, z) + T^{m_1}\phi( x, R_1(y, z) ) , R_1(x, R_1(y, z) ), \\
& ( b +  T^{m_2}\phi( y, z) + T^{m_1}\phi(R_2(x, R_1(y, z)),   R_2(y, z) ),R_1(R_2(x, R_1(y, z)),  R_2(y, z))), \\
& ( a + T^{m_2}\phi(x, R_1(y, z) ) +   T^{m_2}\phi(R_2(x, R_1(y, z)),   R_2(y, z)),R_2(R_2(x, R_1(y, z)),  R_2(y, z))) ~) .
\end{align*}
Given that $(V,S)$ is a Yang-Baxter set, we obtain the following three equations
\begin{align*}	
T^{m_1}\phi( R_2( x,y) , z) +	
			T^{m_1}\phi(R_1(x,  y ), R_1( R_2( x, y), z) )=& T^{m_1}\phi( y, z) + T^{m_1}\phi(x, R_1(y, z) )\\
T^{m_1}\phi(x, y) + T^{m_2}\phi( R_1(x,y),R_1( R_2( x, y), z) ) =&   T^{m_2}\phi( y, z) + T^{m_1}\phi(R_2(x, R_1(y, z)),R_2(y, z) ) \\	
T^{m_2} \phi(x, y) + T^{m_2}\phi( R_2( x, y) , z) =&  T^{m_2}\phi(x, R_1(y, z) )+ T^{m_2}\phi(R_2(x, R_1(y, z)),   R_2(y, z))
\end{align*}
and by adding these equalities the result follows.
\end{proof}


\section{Twisted biquandle Cocycle invariants of classical knots}\label{sec:twisted_biquandle_cocycle_invariants_of_classical_knots}

Let $K$ be a simple, closed, oriented, smooth curve with normals on a plane. This curve $K$ divides the plane into regions. Fix one of the regions as a {\it base region}, and assign integer $0$ to it. Now let $\mathcal{R}$ be any region. We will assign an integer to $\mathcal{R}$, denoted by $\mathcal{L}(\mathcal{R})$, termed the {\it Alexander numbering} of $\mathcal{R}$.

Consider a smooth arc $\alpha$ on the plane from a point lying in the interior of the base region to $\mathcal{R}$, such that the intersection points of $\alpha$ with $K$ are only transversal double points. Suppose that while tracing the curve $\alpha$ to the region $\mathcal{R}$, it intersects $K$ at $n_1$ points where the normal points in the direction of tracing $\alpha$ and at $n_2$ points where the normal points in the opposite direction of tracing $\alpha$. Then $\mathcal{L}(\mathcal{R})$ is $n_1-n_2$.

The Alexander numbering does not depend on the choice of $\alpha$ but depends on the choice of the base region. For more on Alexander numbering and its relation to knots, we refer the reader to \cite{MR1885217,MR1695171,MR1487374}

\begin{definition}
Let $K$ be an oriented classical knot diagram with normals. Let $\tau$ be a crossing. There are four regions near $\tau$, and the unique region from which the normals of over-arc and under-arc point is called the {\it source region} of $\tau$.
\end{definition}

\begin{definition}
The {\it Alexander numbering} $\mathcal{L}(\tau)$ of a crossing $\tau$  is defined to be $\mathcal{L}(\mathcal{R})$ where $\mathcal{R}$ is the source region of $\tau$. See Fig \ref{fig:alexander_numbering_crossing} for illustration.
\end{definition}

\begin{figure}[H]
\begin{center}
\includegraphics[height=1in,width=1.5in,angle=00]{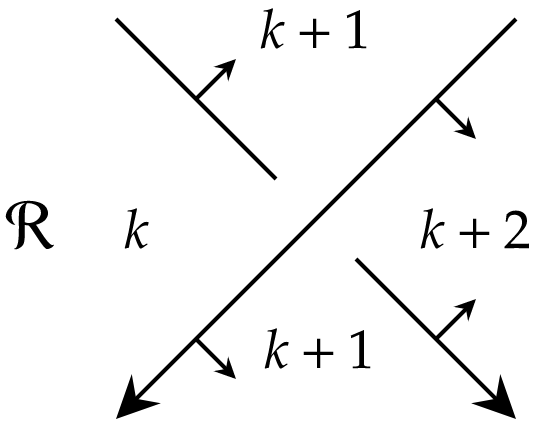}
\end{center}
\caption{The Alexander numbering of a crossing $\tau$.}
\label{fig:alexander_numbering_crossing}
\end{figure}

\begin{figure}[H]
\begin{center}
\includegraphics[height=1.5in,width=1.5in,angle=00]{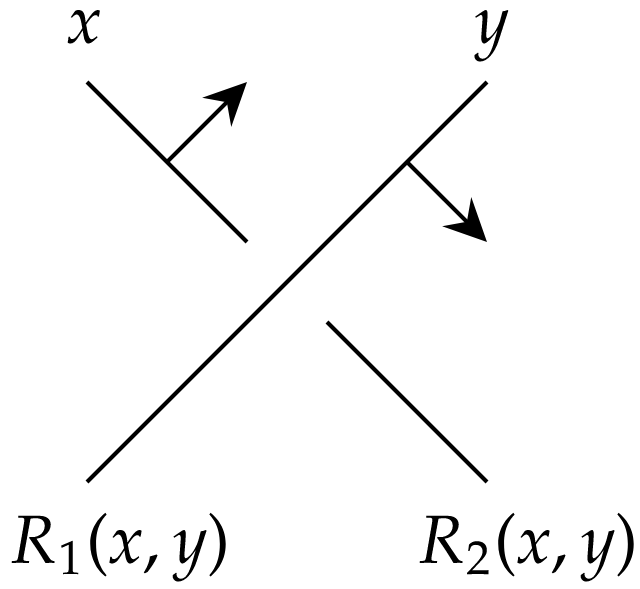}
\end{center}
\caption{Coloring of arcs by the elements of a twisted biquandle $(X,f,R)$.}
\label{fig:coloring_rules}
\end{figure}

\begin{figure}[H]
\tikzset{every picture/.style={line width=0.75pt}} 

\begin{tikzpicture}[x=0.75pt,y=0.75pt,yscale=-1,xscale=1]

\draw    (150,81) -- (52.12,178.88) ;
\draw [shift={(50,181)}, rotate = 315] [fill={rgb, 255:red, 0; green, 0; blue, 0 }  ][line width=0.08]  [draw opacity=0] (10.72,-5.15) -- (0,0) -- (10.72,5.15) -- (7.12,0) -- cycle    ;
\draw    (50,81) -- (90,121) ;
\draw    (110,141) -- (147.88,178.88) ;
\draw [shift={(150,181)}, rotate = 225] [fill={rgb, 255:red, 0; green, 0; blue, 0 }  ][line width=0.08]  [draw opacity=0] (10.72,-5.15) -- (0,0) -- (10.72,5.15) -- (7.12,0) -- cycle    ;
\draw    (70,101) -- (77.88,93.12) ;
\draw [shift={(80,91)}, rotate = 135] [fill={rgb, 255:red, 0; green, 0; blue, 0 }  ][line width=0.08]  [draw opacity=0] (5.36,-2.57) -- (0,0) -- (5.36,2.57) -- (3.56,0) -- cycle    ;
\draw    (130,101) -- (137.88,108.88) ;
\draw [shift={(140,111)}, rotate = 225] [fill={rgb, 255:red, 0; green, 0; blue, 0 }  ][line width=0.08]  [draw opacity=0] (5.36,-2.57) -- (0,0) -- (5.36,2.57) -- (3.56,0) -- cycle    ;
\draw    (130,161) -- (137.88,153.12) ;
\draw [shift={(140,151)}, rotate = 135] [fill={rgb, 255:red, 0; green, 0; blue, 0 }  ][line width=0.08]  [draw opacity=0] (5.36,-2.57) -- (0,0) -- (5.36,2.57) -- (3.56,0) -- cycle    ;
\draw    (70,161) -- (77.88,168.88) ;
\draw [shift={(80,171)}, rotate = 225] [fill={rgb, 255:red, 0; green, 0; blue, 0 }  ][line width=0.08]  [draw opacity=0] (5.36,-2.57) -- (0,0) -- (5.36,2.57) -- (3.56,0) -- cycle    ;
\draw    (350,81) -- (447.88,178.88) ;
\draw [shift={(450,181)}, rotate = 225] [fill={rgb, 255:red, 0; green, 0; blue, 0 }  ][line width=0.08]  [draw opacity=0] (10.72,-5.15) -- (0,0) -- (10.72,5.15) -- (7.12,0) -- cycle    ;
\draw    (410,121) -- (450,81) ;
\draw    (352.12,178.88) -- (390,141) ;
\draw [shift={(350,181)}, rotate = 315] [fill={rgb, 255:red, 0; green, 0; blue, 0 }  ][line width=0.08]  [draw opacity=0] (10.72,-5.15) -- (0,0) -- (10.72,5.15) -- (7.12,0) -- cycle    ;
\draw    (370,101) -- (377.88,93.12) ;
\draw [shift={(380,91)}, rotate = 135] [fill={rgb, 255:red, 0; green, 0; blue, 0 }  ][line width=0.08]  [draw opacity=0] (5.36,-2.57) -- (0,0) -- (5.36,2.57) -- (3.56,0) -- cycle    ;
\draw    (430,161) -- (437.88,153.12) ;
\draw [shift={(440,151)}, rotate = 135] [fill={rgb, 255:red, 0; green, 0; blue, 0 }  ][line width=0.08]  [draw opacity=0] (5.36,-2.57) -- (0,0) -- (5.36,2.57) -- (3.56,0) -- cycle    ;
\draw    (430,101) -- (437.88,108.88) ;
\draw [shift={(440,111)}, rotate = 225] [fill={rgb, 255:red, 0; green, 0; blue, 0 }  ][line width=0.08]  [draw opacity=0] (5.36,-2.57) -- (0,0) -- (5.36,2.57) -- (3.56,0) -- cycle    ;
\draw    (370,161) -- (377.88,168.88) ;
\draw [shift={(380,171)}, rotate = 225] [fill={rgb, 255:red, 0; green, 0; blue, 0 }  ][line width=0.08]  [draw opacity=0] (5.36,-2.57) -- (0,0) -- (5.36,2.57) -- (3.56,0) -- cycle    ;

\draw (37,63.4) node [anchor=north west][inner sep=0.75pt]    {$x$};
\draw (157,63.4) node [anchor=north west][inner sep=0.75pt]    {$y$};
\draw (17,183.4) node [anchor=north west][inner sep=0.75pt]    {$R_{1}( x,\ y)$};
\draw (138,183.4) node [anchor=north west][inner sep=0.75pt]    {$R_{2}( x,\ y)$};
\draw (301,186.4) node [anchor=north west][inner sep=0.75pt]    {$\overline{R_{1}}( a,\ b)$};
\draw (422,186.4) node [anchor=north west][inner sep=0.75pt]    {$\overline{R_{2}}( a,\ b)$};
\draw (297,19.4) node [anchor=north west][inner sep=0.75pt]    {$R_{1}( x,\ y)$};
\draw (418,19.4) node [anchor=north west][inner sep=0.75pt]    {$R_{2}( x,\ y)$};
\draw (357.25,211.74) node [anchor=north west][inner sep=0.75pt]  [rotate=-89.92]  {$=$};
\draw (467.57,212) node [anchor=north west][inner sep=0.75pt]  [rotate=-89.92]  {$=$};
\draw (347,233.4) node [anchor=north west][inner sep=0.75pt]    {$x$};
\draw (458,233.4) node [anchor=north west][inner sep=0.75pt]    {$y$};
\draw (357.57,41) node [anchor=north west][inner sep=0.75pt]  [rotate=-89.92]  {$=$};
\draw (343.04,62.4) node [anchor=north west][inner sep=0.75pt]    {$a$};
\draw (464.6,41) node [anchor=north west][inner sep=0.75pt]  [rotate=-89.92]  {$=$};
\draw (451,62.4) node [anchor=north west][inner sep=0.75pt]    {$b$};
\draw (36,122.4) node [anchor=north west][inner sep=0.75pt]    {$\mathcal{L}( \tau _{1})$};
\draw (341,122.4) node [anchor=north west][inner sep=0.75pt]    {$\mathcal{L}( \tau _{2})$};
\draw (42,270.4) node [anchor=north west][inner sep=0.75pt]    {$B_{n}( \tau _{_{1}} ,\ \mathcal{C}) =T^{-n\mathcal{L}( \tau _{1})}( \phi ( x,y))$};
\draw (361,270.4) node [anchor=north west][inner sep=0.75pt]    {$B_{n}( \tau _{2} ,\ \mathcal{C}) =T^{-n\mathcal{L}( \tau _{_{2}})}\left( \phi ( x,y)^{-1}\right)$};

\end{tikzpicture}
\caption{Associated Boltzmann weight to the positive and negative crossings.}
\label{fig:associated_virtua_boltzman_weight}

\end{figure}
Consider a classical knot diagram $K$, a finite twisted biquandle $(X,f,R)$, a finite $\Lambda$-module $M$, and $n$ as a fixed positive integer. We use multiplicative notation instead of addition for the elements of $M$. Let $\phi \in Z^2_{\TBQ, (n)}(X;M)$ and $\mathcal{C}$ be a coloring of $K$ using $X$ under the coloring rules shown in Fig \ref{fig:coloring_rules}. A {\it Boltzmann weight} $B_{n}(\tau, C)$ at a crossing $\tau$ is defined as follows:

Let $\mathfrak{u}$ be the under-arc away from which normal to the over-arc points. Let $\mathfrak{o}$ be the over-arc towards which the normal to the under-arc points. Let $\mathcal{C}(\mathfrak{u})=x$ and $\mathcal{C}(\mathfrak{o})=y$. Then define $$B_{n}(\tau, \mathcal{C})=T^{-n\mathcal{L}(\tau)}\big(\phi(x,y)^{\epsilon(\tau)}\big),$$ where $\epsilon(\tau)$ is $1$ if $\tau$ is a positive crossing and $-1$ if $\tau$ is a negative crossing. In Figure \ref{fig:associated_virtua_boltzman_weight}, the association of Boltzmann weight at positive and negative crossings is illustrated.

The {\it state-sum} or the {\it partition function} associated to $K$ is given by
\[
\Phi_n(K)= \sum_{\mathcal{C}} \prod_{\tau} B_{n}(\tau, \mathcal{C}),
\]
where the product $\prod_{\tau} B_{n}(\tau, \mathcal{C})$ is taken over all crossings $\tau$ of the given diagram, and the sum is taken over all the possible colorings of $K$ using $X$ under the rules depicted in Figure \ref{fig:coloring_rules}. The state-sum $\Phi_n(K)$ is in the integral group ring $\mathbb{Z}[M]$.

\begin{theorem}
Let $K_1$ and $K_2$ be two knot diagrams representing the same knot. Then the state-sum invariant is well defined up to the action of $\mathbb{Z}=\langle T \rangle$. In other words $\Phi_n(K_1)= T^k\Phi_n(K_2)$ for some $k \in \mathbb{Z}$.
\end{theorem}
\begin{proof}
We need to show the invariance of the state-sum under the Reidemeister moves (see Figure \ref{fig:Reidemeister_moves}.
For every $x \in X$, there exists a unique $a \in X$ such that $R(a,x)=(a,x)$. Since $\phi$ is  $2$-cocycle, thus $\phi(a,x)=\phi(x,a)=1 \in M$. Noting that $T$ acts on $M$ via an automorphism, we observe that performing the RI move does not alter the state-sum.

\begin{figure}[H]
\begin{center}
\includegraphics[height=1.5in,width=4in,angle=00]{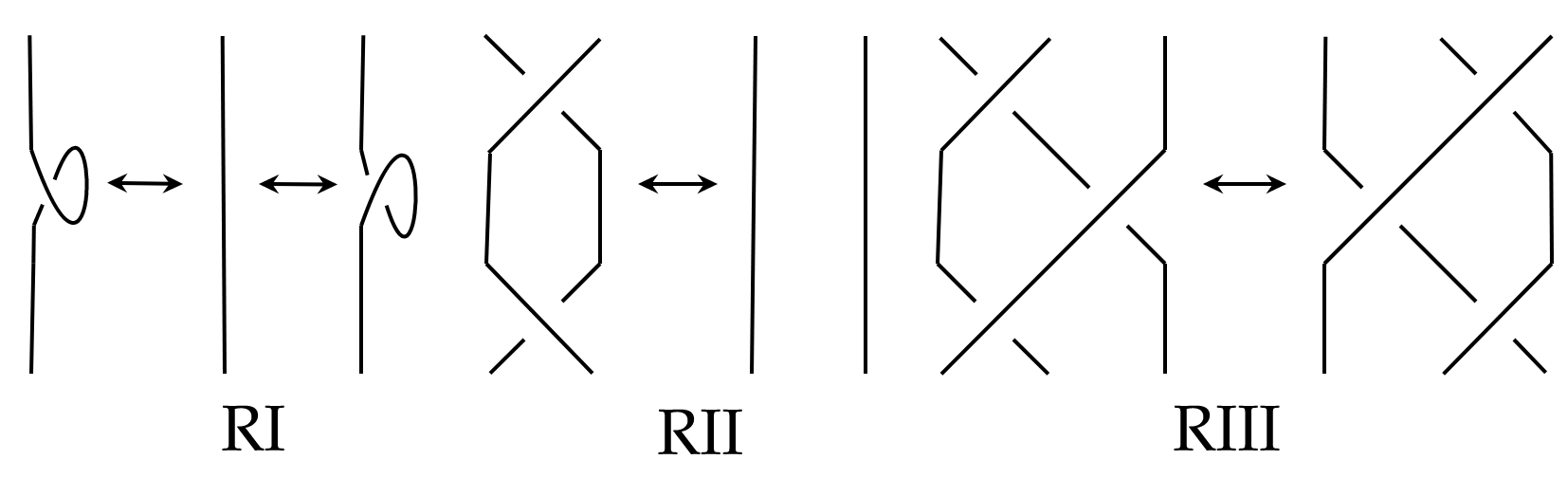}
\end{center}
\caption{Reidemeister moves.}
\label{fig:Reidemeister_moves}

\begin{center}
\includegraphics[height=1.5in,width=3in,angle=00]{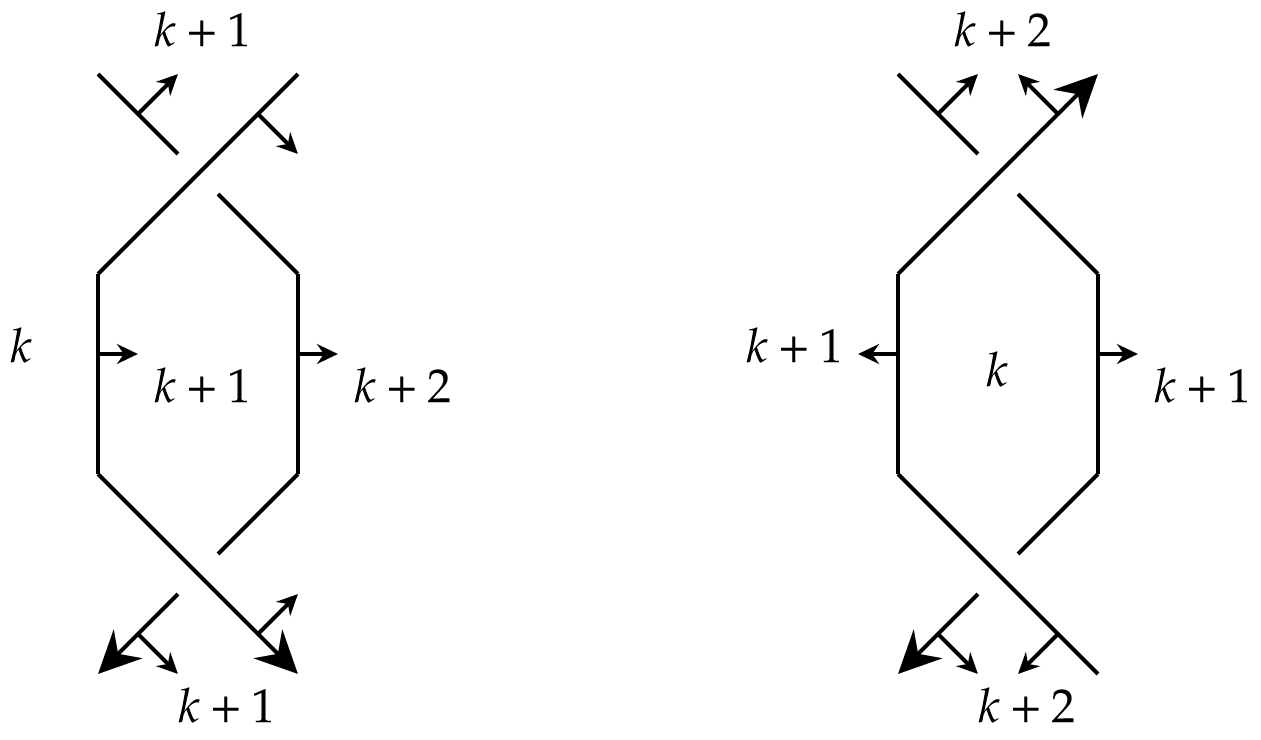}
\end{center}
\caption{Type RII-moves and Alexander numbering.}
\label{fig:type_RII_moves_and_Alexander_numbering.png}
\end{figure}
For the RII-moves shown in Figure \ref{fig:type_RII_moves_and_Alexander_numbering.png}, the signs of the crossings are opposite and the region contributing to the Boltzmann weights for the crossings in each move is the same. Consequently, the product of Boltzmann's weights associated with the crossings in RII-move takes the form $T^{-nk}(\phi(x,y)^{\epsilon}) T^{-nk} (\phi(x,y)^{-\epsilon})$ which simplifies to the identity element in $M$. Thus the state-sum is invariant under type II moves.

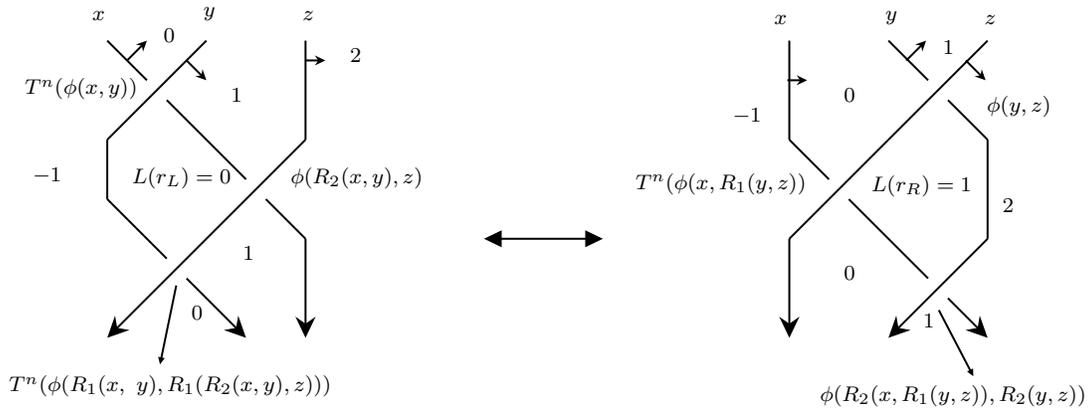
\begin{figure}[H]

\tikzset{every picture/.style={line width=0.75pt}} 

\begin{tikzpicture}[x=0.75pt,y=0.75pt,yscale=-1,xscale=1]

\draw    (100,50) -- (120,70) ;
\draw    (150,50) -- (100,100) ;
\draw    (130,80) -- (170,120) ;
\draw    (100,100) -- (100,130) ;
\draw    (200,50) -- (200,100) ;
\draw    (200,100) -- (102.12,197.88) ;
\draw [shift={(100,200)}, rotate = 315] [fill={rgb, 255:red, 0; green, 0; blue, 0 }  ][line width=0.08]  [draw opacity=0] (10.72,-5.15) -- (0,0) -- (10.72,5.15) -- (7.12,0) -- cycle    ;
\draw    (100,130) -- (130,160) ;
\draw    (140,170) -- (167.88,197.88) ;
\draw [shift={(170,200)}, rotate = 225] [fill={rgb, 255:red, 0; green, 0; blue, 0 }  ][line width=0.08]  [draw opacity=0] (10.72,-5.15) -- (0,0) -- (10.72,5.15) -- (7.12,0) -- cycle    ;
\draw    (180,130) -- (200,150) ;
\draw    (200,150) -- (200,197) ;
\draw [shift={(200,200)}, rotate = 270] [fill={rgb, 255:red, 0; green, 0; blue, 0 }  ][line width=0.08]  [draw opacity=0] (10.72,-5.15) -- (0,0) -- (10.72,5.15) -- (7.12,0) -- cycle    ;
\draw    (444,50) -- (444,100) ;
\draw    (544,50) -- (494,100) ;
\draw    (494,50) -- (514,70) ;
\draw    (524,80) -- (544,100) ;
\draw    (444,100) -- (464,120) ;
\draw    (494,100) -- (444,150) ;
\draw    (444,150) -- (444,197) ;
\draw [shift={(444,200)}, rotate = 270] [fill={rgb, 255:red, 0; green, 0; blue, 0 }  ][line width=0.08]  [draw opacity=0] (10.72,-5.15) -- (0,0) -- (10.72,5.15) -- (7.12,0) -- cycle    ;
\draw    (474,130) -- (514,170) ;
\draw    (544,100) -- (544,150) ;
\draw    (544,150) -- (496.12,197.88) ;
\draw [shift={(494,200)}, rotate = 315] [fill={rgb, 255:red, 0; green, 0; blue, 0 }  ][line width=0.08]  [draw opacity=0] (10.72,-5.15) -- (0,0) -- (10.72,5.15) -- (7.12,0) -- cycle    ;
\draw    (524,180) -- (541.88,197.88) ;
\draw [shift={(544,200)}, rotate = 225] [fill={rgb, 255:red, 0; green, 0; blue, 0 }  ][line width=0.08]  [draw opacity=0] (10.72,-5.15) -- (0,0) -- (10.72,5.15) -- (7.12,0) -- cycle    ;
\draw    (110,60) -- (117.88,52.12) ;
\draw [shift={(120,50)}, rotate = 135] [fill={rgb, 255:red, 0; green, 0; blue, 0 }  ][line width=0.08]  [draw opacity=0] (5.36,-2.57) -- (0,0) -- (5.36,2.57) -- (3.56,0) -- cycle    ;
\draw    (140,60) -- (147.88,67.88) ;
\draw [shift={(150,70)}, rotate = 225] [fill={rgb, 255:red, 0; green, 0; blue, 0 }  ][line width=0.08]  [draw opacity=0] (5.36,-2.57) -- (0,0) -- (5.36,2.57) -- (3.56,0) -- cycle    ;
\draw    (200,60) -- (207,60) ;
\draw [shift={(210,60)}, rotate = 180] [fill={rgb, 255:red, 0; green, 0; blue, 0 }  ][line width=0.08]  [draw opacity=0] (5.36,-2.57) -- (0,0) -- (5.36,2.57) -- (3.56,0) -- cycle    ;
\draw    (443,70) -- (450,70) ;
\draw [shift={(453,70)}, rotate = 180] [fill={rgb, 255:red, 0; green, 0; blue, 0 }  ][line width=0.08]  [draw opacity=0] (5.36,-2.57) -- (0,0) -- (5.36,2.57) -- (3.56,0) -- cycle    ;
\draw    (503,60) -- (510.88,52.12) ;
\draw [shift={(513,50)}, rotate = 135] [fill={rgb, 255:red, 0; green, 0; blue, 0 }  ][line width=0.08]  [draw opacity=0] (5.36,-2.57) -- (0,0) -- (5.36,2.57) -- (3.56,0) -- cycle    ;
\draw    (533,60) -- (540.88,67.88) ;
\draw [shift={(543,70)}, rotate = 225] [fill={rgb, 255:red, 0; green, 0; blue, 0 }  ][line width=0.08]  [draw opacity=0] (5.36,-2.57) -- (0,0) -- (5.36,2.57) -- (3.56,0) -- cycle    ;
\draw    (134.8,173.6) -- (127.21,210.46) ;
\draw [shift={(126.6,213.4)}, rotate = 281.64] [fill={rgb, 255:red, 0; green, 0; blue, 0 }  ][line width=0.08]  [draw opacity=0] (3.57,-1.72) -- (0,0) -- (3.57,1.72) -- cycle    ;
\draw    (519.4,186) -- (535.41,216.74) ;
\draw [shift={(536.8,219.4)}, rotate = 242.48] [fill={rgb, 255:red, 0; green, 0; blue, 0 }  ][line width=0.08]  [draw opacity=0] (3.57,-1.72) -- (0,0) -- (3.57,1.72) -- cycle    ;
\draw    (293,150) -- (347,150) ;
\draw [shift={(350,150)}, rotate = 180] [fill={rgb, 255:red, 0; green, 0; blue, 0 }  ][line width=0.08]  [draw opacity=0] (8.93,-4.29) -- (0,0) -- (8.93,4.29) -- cycle    ;
\draw [shift={(290,150)}, rotate = 0] [fill={rgb, 255:red, 0; green, 0; blue, 0 }  ][line width=0.08]  [draw opacity=0] (8.93,-4.29) -- (0,0) -- (8.93,4.29) -- cycle    ;

\draw (91,33.4) node [anchor=north west][inner sep=0.75pt]  [font=\scriptsize]  {$x$};
\draw (147,31.4) node [anchor=north west][inner sep=0.75pt]  [font=\scriptsize]  {$y$};
\draw (197,33.4) node [anchor=north west][inner sep=0.75pt]  [font=\scriptsize]  {$z$};
\draw (57,67.4) node [anchor=north west][inner sep=0.75pt]  [font=\scriptsize]  {$T^{n}( \phi ( x,y))$};
\draw (191,112.4) node [anchor=north west][inner sep=0.75pt]  [font=\scriptsize]  {$\phi ( R_{2}( x,y) ,z)$};
\draw (49,217.4) node [anchor=north west][inner sep=0.75pt]  [font=\scriptsize]  {$T^{n}( \phi ( R_{1}( x,\ y) ,R_{1}( R_{2}( x,y) ,z)))$};
\draw (435,35.4) node [anchor=north west][inner sep=0.75pt]  [font=\scriptsize]  {$x$};
\draw (491,33.4) node [anchor=north west][inner sep=0.75pt]  [font=\scriptsize]  {$y$};
\draw (541,35.4) node [anchor=north west][inner sep=0.75pt]  [font=\scriptsize]  {$z$};
\draw (542.2,76.4) node [anchor=north west][inner sep=0.75pt]  [font=\scriptsize]  {$\phi ( y,z)$};
\draw (365,115.4) node [anchor=north west][inner sep=0.75pt]  [font=\scriptsize]  {$T^{n}( \phi ( x,R_{1}( y,z))$};
\draw (458,222.4) node [anchor=north west][inner sep=0.75pt]  [font=\scriptsize]  {$\phi ( R_{2}( x,R_{1}( y,z)) ,R_{2}( y,z))$};
\draw (61,112.4) node [anchor=north west][inner sep=0.75pt]  [font=\scriptsize]  {$-1$};
\draw (127,42.4) node [anchor=north west][inner sep=0.75pt]  [font=\scriptsize]  {$0$};
\draw (161,72.4) node [anchor=north west][inner sep=0.75pt]  [font=\scriptsize]  {$1$};
\draw (221,52.4) node [anchor=north west][inner sep=0.75pt]  [font=\scriptsize]  {$2$};
\draw (167,152.4) node [anchor=north west][inner sep=0.75pt]  [font=\scriptsize]  {$1$};
\draw (141,182.4) node [anchor=north west][inner sep=0.75pt]  [font=\scriptsize]  {$0$};
\draw (414,82.4) node [anchor=north west][inner sep=0.75pt]  [font=\scriptsize]  {$-1$};
\draw (470,72.4) node [anchor=north west][inner sep=0.75pt]  [font=\scriptsize]  {$0$};
\draw (520,48.4) node [anchor=north west][inner sep=0.75pt]  [font=\scriptsize]  {$1$};
\draw (111,112.4) node [anchor=north west][inner sep=0.75pt]  [font=\scriptsize]  {$L( r_{L}) =0$};
\draw (484,116.4) node [anchor=north west][inner sep=0.75pt]  [font=\scriptsize]  {$L( r_{R}) =1$};
\draw (470,162.4) node [anchor=north west][inner sep=0.75pt]  [font=\scriptsize]  {$0$};
\draw (510,186.4) node [anchor=north west][inner sep=0.75pt]  [font=\scriptsize]  {$1$};
\draw (550,128.4) node [anchor=north west][inner sep=0.75pt]  [font=\scriptsize]  {$2$};

\end{tikzpicture}

\caption{RIII move and Boltzmann weights.}
\label{fig:RIII_move_twisted_virtual_Boltzmann_weight}
\end{figure}

In Figure \ref{fig:RIII_move_twisted_virtual_Boltzmann_weight}, we have RIII move with a specific orientation, where Boltzmann weights are indicated at the crossings. It is noteworthy that in left diagram, the Alexander numbering $\mathcal{L}(r_1)=0$ changes to $\mathcal{L}(r_2)=1$ during the move, while the remaining Alexander numbers remain unchanged. The product of Boltzmann weights in the left of the diagram is
\[L=
T^n( \phi(x,y) )  ~\phi(R_2(x,y),z) ~T^n(\phi( R_1(x, y ), R_1(R_2(x, y), z))),
\]
and the product of Boltzmann weights in the right of the diagram is
\[R=
\phi(y,z)~  T^n(\phi(x, R_1(y, z)))~ \phi(R_2(x, R_1(y, z)), R_2(y,z)).
\]
Since $\phi \in Z^2_{\TBQ, (n)}(X;M)$, thus $L=R$, and hence the state-sum remains unchanged under the illustrated RIII-move. The rest of the cases follow from the combinations with type II moves (see \cite{MR3013186,MR939474}).

Observe that altering the base region in a given knot diagram affects the state-sum through the action of $T^k$ for some $k \in \mathbb{Z}$. This completes the proof.
\end{proof}

\begin{figure}[H]
\begin{center}
\includegraphics[height=2in,width=5.5in,angle=00]{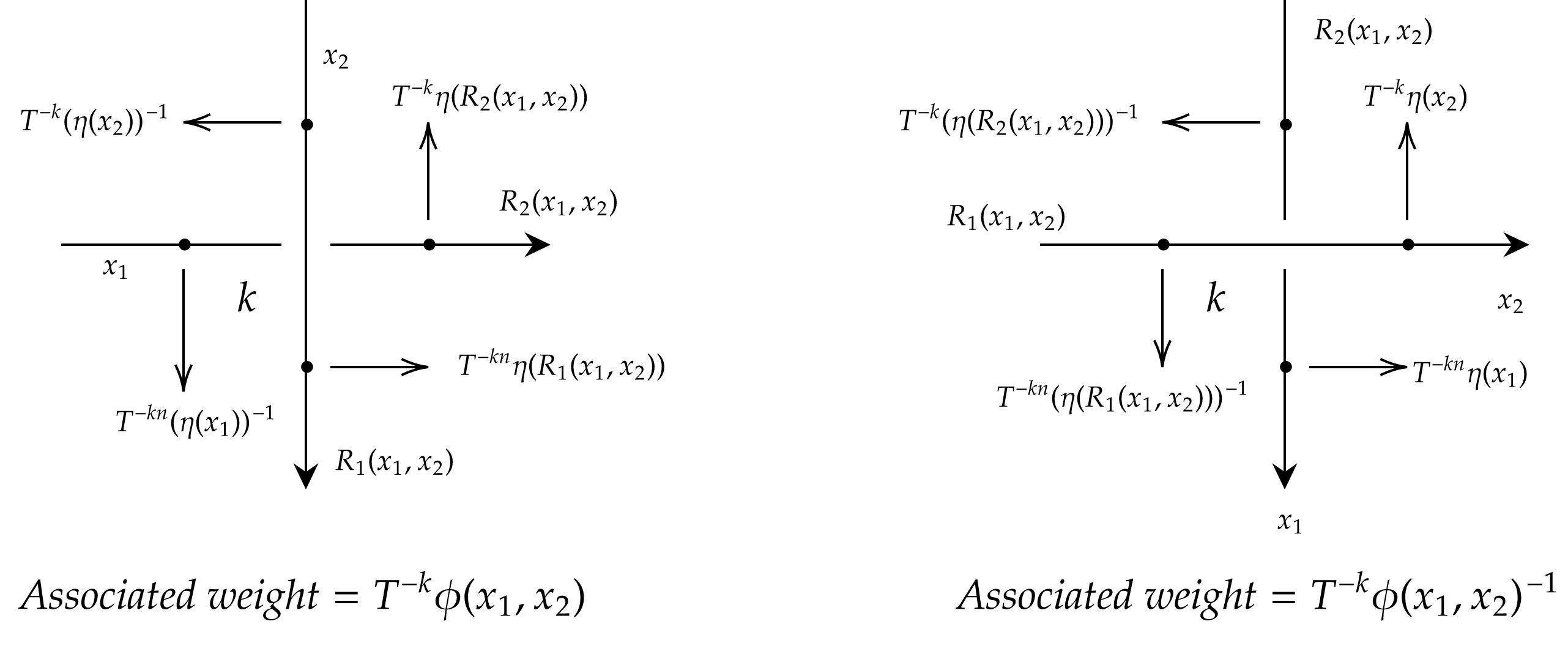}
\end{center}
\caption{Illustration of assigning labels to the end points near each crossing, where $k$ denotes the Alexander numbering of the crossing.}
\label{fig:assigning_of_labels_around_crossings}
\end{figure}

\begin{proposition}\label{prop:cohomologous_cycles_give_the_same}
Let $(X,f,R)$ be a twisted biquandle, $M$ a finite $\Lambda$-module, and $\phi \in Z^2_{\TBQ, (n)}(X;M)$. If $\phi$ is a coboundary, that is $\phi=\delta^2_{(n)}(\eta)$, where $\eta \in C^1_{\TBQ}(X;M)$, then the state-sum $\Phi_n(K)$ is the number of colorings of $K$ by the biquandle $(X,f,R)$.
\end{proposition}
\begin{proof}
For all $x, y \in X$, we have
\[\phi(x, y)=\delta^2_{(0,n)}(\eta)(x, y)=(\eta(y))^{-1} T^n(\eta(R_1(x, y))  \eta(R_2(x, y))  T^n((\eta(x))^{-1}).\]
Now for a knot diagram $K$ and a coloring $\mathcal{C}$, we assign labels around each crossing as illustrated in Figure \ref{fig:assigning_of_labels_around_crossings}. The product of the labels around each crossing equals the Boltzmann weight assigned to that crossing. Moreover, the labels between any two consecutive crossings are multiplicative inverse of each other. Consequently, the product of all the Boltzmann weights in $K$ for the coloring $\mathcal{C}$ is $1$. Therefore, the state-sum $\Phi_n(K)$ of $K$ is the number of colorings of $K$ by $(X,f,R)$.
\end{proof}

\begin{example}
Consider the Hopf link $K$ as depicted in Figure \ref{fig:hopf_link_with_Alexander_numbering}, accompanied by the Alexander numbering of regions. Let $X$ be a twisted quandle $(T_2, (0~1), *)$, where $T_2$ is the trivial quandle with two elements, and $M$ be the module $\Lambda/\langle 1-T^2 \rangle$. As an abelian group, $M$ is generated $1$ and $T$, denoted multiplicatively as $u$ and $v$, respectively. Consider a $2$-cocycle $\phi \in Z^2_{\TBQ, (1)}(X;M)$ given by $\phi = u^{\chi(0,1)}v^{\chi(1,0)}$. Then, the state-sum invariant is $\Phi_1(K)=2 + 2uv$. Note that $T\Phi_1(K)=\Phi_1(K)$.
\begin{figure}[H]
\begin{center}
\includegraphics[height=1in,width=2in,angle=00]{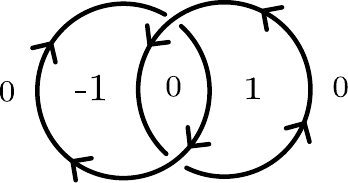}
\end{center}
\caption{Hopf link with Alexander numbering.}
\label{fig:hopf_link_with_Alexander_numbering}
\end{figure}
\end{example}

\section{Twisted biquandle cocycle invariants of knotted surfaces}\label{sec:twisted_biquandle_cocycle_invariants_of_knotted_surfaces}
In this section, we define the state-sum invariant of knotted surfaces utilizing broken surface diagrams. 

A knotted surface is defined as a smooth embedding of an orientable closed surface in $\mathbb{R}^4$. Analogous to knot diagrams, a knotted surface can be represented by its generic projection onto $\mathbb{R}^3$, including relative height information. Such projection diagrams are called {\it broken surface diagrams}. Locally, these diagrams are shown in Figure \ref{fig:broken_surface_diagrams}, illustrating double curve, triple point and isolated branch point.

\begin{figure}[H]
\begin{center}
\includegraphics[height=3in,width=3in,angle=00]{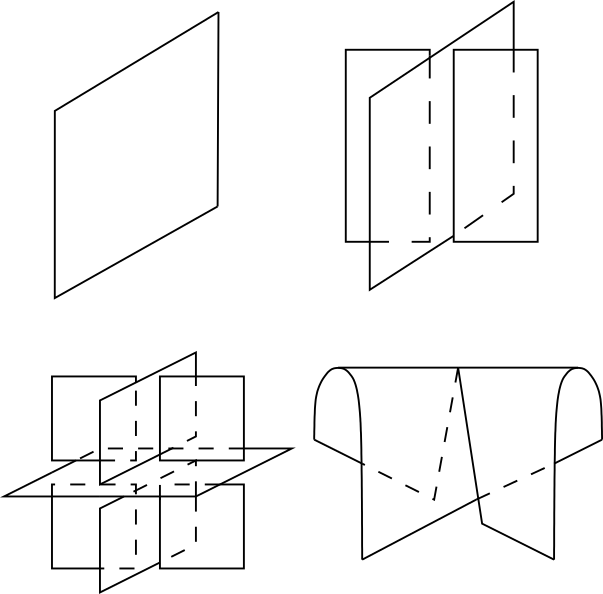}
\end{center}
\caption{Local pictures of broken surface diagrams}
\label{fig:broken_surface_diagrams}
\end{figure}

Similar to knot diagrams, broken surface diagrams are employed to define invariants of knotted surfaces. Examples of such invariants include, for instance, quandle colorings, biquandle colorings, state sum invariants \cite{MR1990571,MR3868945}, and fundamental biquandles \cite{MR3190125}.

For a given finite twisted biquandle $(X,f,R)$, the coloring rule of a broken surface diagram is defined using normals, as depicted in Figure \ref{fig:coloring_rule_surface}.

\begin{figure}[H]
\begin{center}
\includegraphics[height=1.5in,width=2in,angle=00]{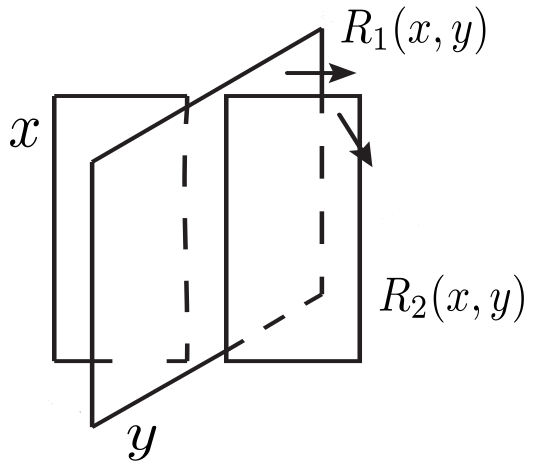}
\end{center}
\caption{Coloring rule of broken surface diagrams.}
\label{fig:coloring_rule_surface}
\end{figure}

\begin{figure}[H]
\begin{center}
\includegraphics[height=2in,width=6in,angle=00]{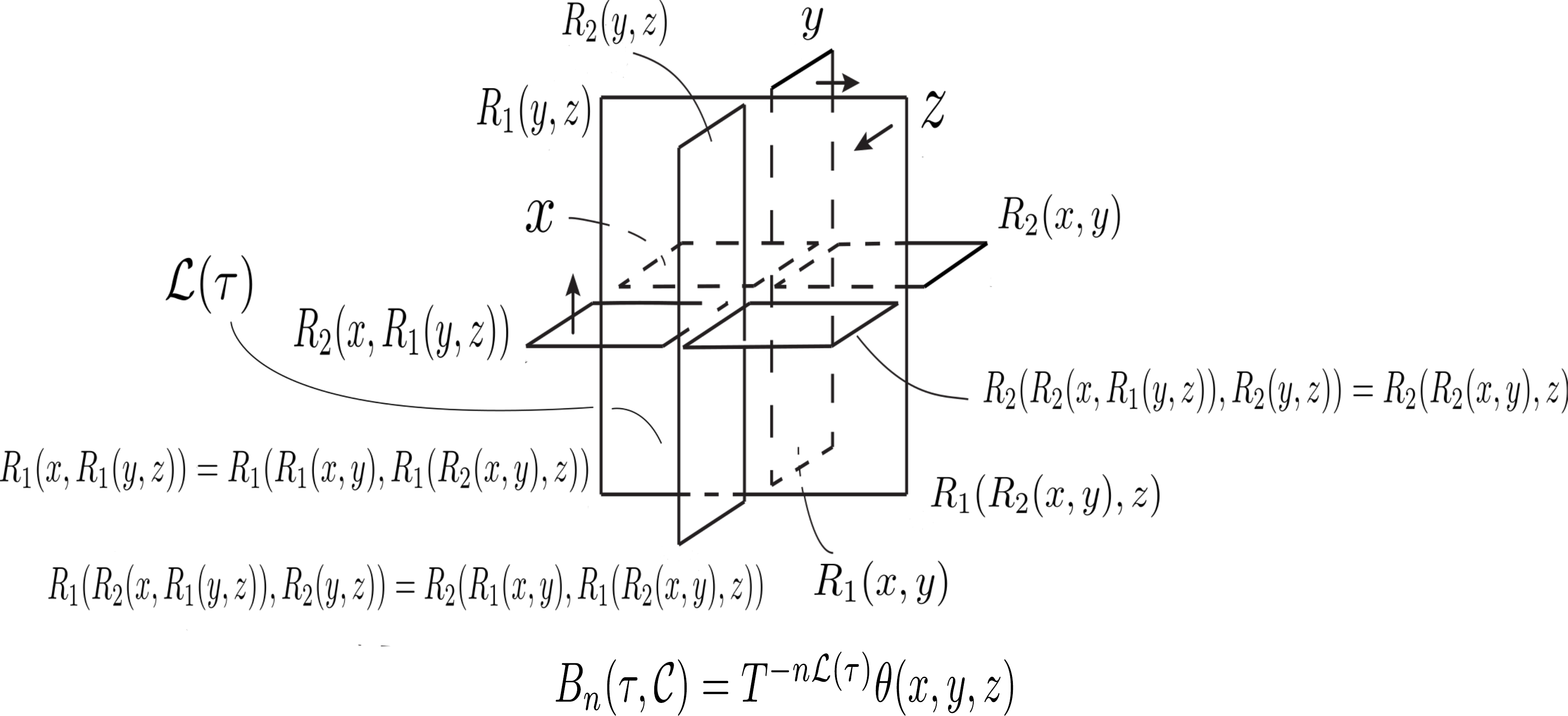}
\end{center}
\caption{Assignment of Boltzmann weight to the triple point $\tau$.}
\label{fig:knotted_weight}
\end{figure}

Analogous to the case of knots, the state sum invariant is defined as follows: Let $K$ be a broken surface diagram and $\mathcal{C}$ be a coloring of $L$ by twisted biquandle $(X,f,R)$. For a triple point $\tau$ in $K$, the source region $R$ and the Alexander numbering $\mathcal{L}(\tau)=\mathcal{L}(R)$ are defined similar to knots (see \cite{MR1885217, MR1487374} for more details). Let  $\theta \in Z^3_{\TBQ, (n)}(X;M)$ be a $3$-cocycle. Then, for  each triple point $\tau$ with Alexander number $\mathcal{L}(\tau)$, assign a weight $B_{n}(\tau, \mathcal{C})=T^{-n\mathcal{L}(\tau)} \theta(x,y,z)^{\epsilon(\tau)}$, where $\epsilon(\tau)$ is the sign of the triple point $\tau$ (see \cite{MR1487374} for details). An illustration is provided in Figure \ref{fig:knotted_weight}. The state sum is defined as follows: \[\Phi_n(K)=\sum_{\mathcal{C}}\prod_{\tau}B_{n}(\tau, \mathcal{C}).\]

By checking the invariance of $\Phi_n(K)$ under the Roseman moves for broken surface diagrams, we obtain the following.

\begin{theorem}
Let $K_1$ and $K_2$ be broken surface diagrams representing the same knotted surface. Then the state-sum invariant is well defined up to the action of $\mathbb{Z}=\langle T \rangle$. In other words $\Phi_n(K_1)= T^k\Phi_n(K_2)$ for some $k \in \mathbb{Z}$.
\end{theorem}

The proof of the following proposition employs a similar argument to that of Proposition \ref{prop:cohomologous_cycles_give_the_same}.
\begin{proposition}
Let $(X,f,R)$ be a twisted biquandle, $M$ a finite $\Lambda$-module, and $\phi \in Z^3_{\TBQ, (n)}(X;M)$. If $\phi$ is a coboundary, that is $\phi=\delta^3_{(n)}(\eta)$, where $\eta \in C^2_{\TBQ}(X;M)$, then the state-sum $\Phi_n(K)$, of a broken surface diagram $K$, is the number of colorings of $K$ by $(X,f,R)$.
\end{proposition}


\section{Computing state-sum using braid charts}\label{sec:computations_using_braid_charts}
In this section, we introduce a method for calculating the state-sum invariant of oriented surface knot $K$ using braid charts, which represents a surface braid whose closure is $K$. This method is inspired and shares similarities with the approach outlined in \cite{MR1990571}.

Let $S$ a surface braid, and let $\Gamma$ a braid chart representing $S$. Let $Q(S)$ denote the fundamental quandle associated with $S$. For detailed instructions on composing a presentation of $Q(S)$ through braid charts, we refer to \cite[Section 10]{MR1990571}. Each white vertex in $\Gamma$ corresponds to a triple point in $S$. These vertices are adjacent to six edges, with labels alternating between $p$ and $p+1$ for some $p$. Additionally, we label each white vertex with the braid index of the corresponding triple point as depicted in Figure \ref{fig:braid_index}.

\begin{figure}[H]
\begin{center}

\includegraphics[height=2in,width=4in,angle=00]{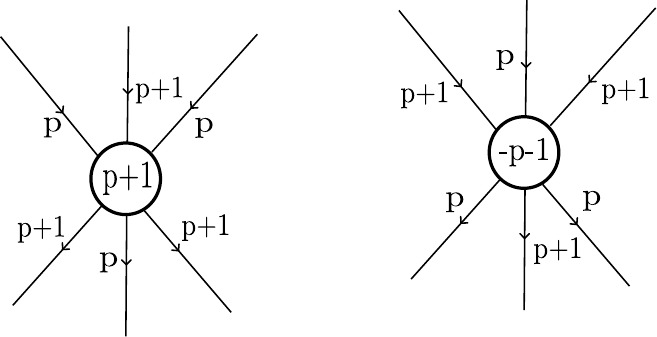}
\end{center}
\caption{Braid index of a white vertex.}
\label{fig:braid_index}
\end{figure}
Let $X=(Y,f,*)$ be a finite twisted quandle, and $\theta \in Z^3_{\TBQ, (n)}(X;M)$, where $M$ a finite $\Lambda$-module. Suppose $\mathcal{C}: Q(S) \to Y$ represents a quandle homomorphism. Then, for each white vertex $W$ with braid index $b(W)$, assign a {\it Boltzmann weight}
\[
B_n(W, \mathcal{C})= T^{-nb(W)} \theta(x, y,z)^{\epsilon(W)},
\]
where $\epsilon(W)$ denotes the sign of $W$, and $(x,y,z)$ is the quandle triple for $W$. Now, define the state-sum of $\Gamma$ as
\[
\Phi_n(\Gamma)= \sum_{\mathcal{C}} \prod_{W} B_n(W, \mathcal{C}),
\]
where $W$ varies over all the white vertices in $\Gamma$, and $\mathcal{C}$ varies over all the quandle homomorphisms from $Q(S)$ to $Y$.

The following results relates state-sum of the chart $\Gamma$ to the state-sum of the knotted surface $\hat{S}.$

\begin{theorem}\label{thm:surface_braid}
Let $S$ be a surface braid represented by a chart $\Gamma$ and let $\hat S$ be the closure of $S$. Then $\Phi_n(\hat{S})=\Phi_n(\Gamma)$.
\end{theorem}
\begin{proof}
The proof follows from Lemma \cite[Lemma 11.1]{MR1990571} and the observation that the braid index of a white vertex is the Alexander numbering of the corresponding triple point in $\hat{S}$ (see \cite[Section 3]{MR1695171}).
\end{proof}
\begin{example}
Consider the braid chart $\Gamma$ depicted in Figure \ref{fig:braid_chart}, representing a surface braid $S$ whose closure forms the $2$-twist spun trefoil. In \cite[Section 11]{MR1990571} it is shown that the fundamental quandle $Q(S)$ of $S$ is
\[
\langle x_1, x_2 ~|~ x_2=(x_1*x_2)*x_1, ~x_2=(x_2*x_1)*x_1 \rangle,
\]
where the quandle triples of the white vertices $W_1, \ldots, W_6$ are 
\[
(x_1*x_2, x_1, x_2), ~(x_1*x_2, x_2, x_1*x_2), ~(x_2, x_1*x_2, x_1), (x_1, x_1*x_2, x_2), (x_1*x_2, x_1, x_1*x_2), (x_1, x_2, x_1*x_2), 
\]
respectively. The sign of the white vertices are
\[
\epsilon(W_1)=\epsilon(W_2)=\epsilon(W_3)=+1,~~\epsilon(W_4)=\epsilon(W_5)=\epsilon(W_6)=-1.
\]
Moreover, the numbers in the white vertices represents their corresponding braid index. By Theorem \ref{thm:surface_braid}, for $\theta \in Z^3_{\TBQ, (n)}(X;M)$
\begin{equation}
    \begin{split}
\Phi_n(\hat S)=&\sum_{\mathcal{C}} T^{-2n}\theta(\mathcal{C}(x_1*x_2), \mathcal{C}(x_1), \mathcal{C}(x_2))\times T^{-2n}\theta(\mathcal{C}(x_1*x_2), \mathcal{C}(x_2), \mathcal{C}(x_1*x_2))\\ &\times T^{-3n} \theta(\mathcal{C}(x_2), \mathcal{C}(x_1*x_2), \mathcal{C}(x_1)) \times T^{3n}\theta(\mathcal{C}(x_1), \mathcal{C}(x_1*x_2), \mathcal{C}(x_2))^{-1} \\ &\times T^{2n} \theta(\mathcal{C}(x_1*x_2), \mathcal{C}(x_1), \mathcal{C}(x_1*x_2))^{-1} \times T^{2n} \theta(\mathcal{C}(x_1), \mathcal{C}(x_2), \mathcal{C}(x_1*x_2))^{-1}
\end{split}
\end{equation}

\begin{figure}[H]
\begin{center}

\includegraphics[height=4.5in,width=4in,angle=00]{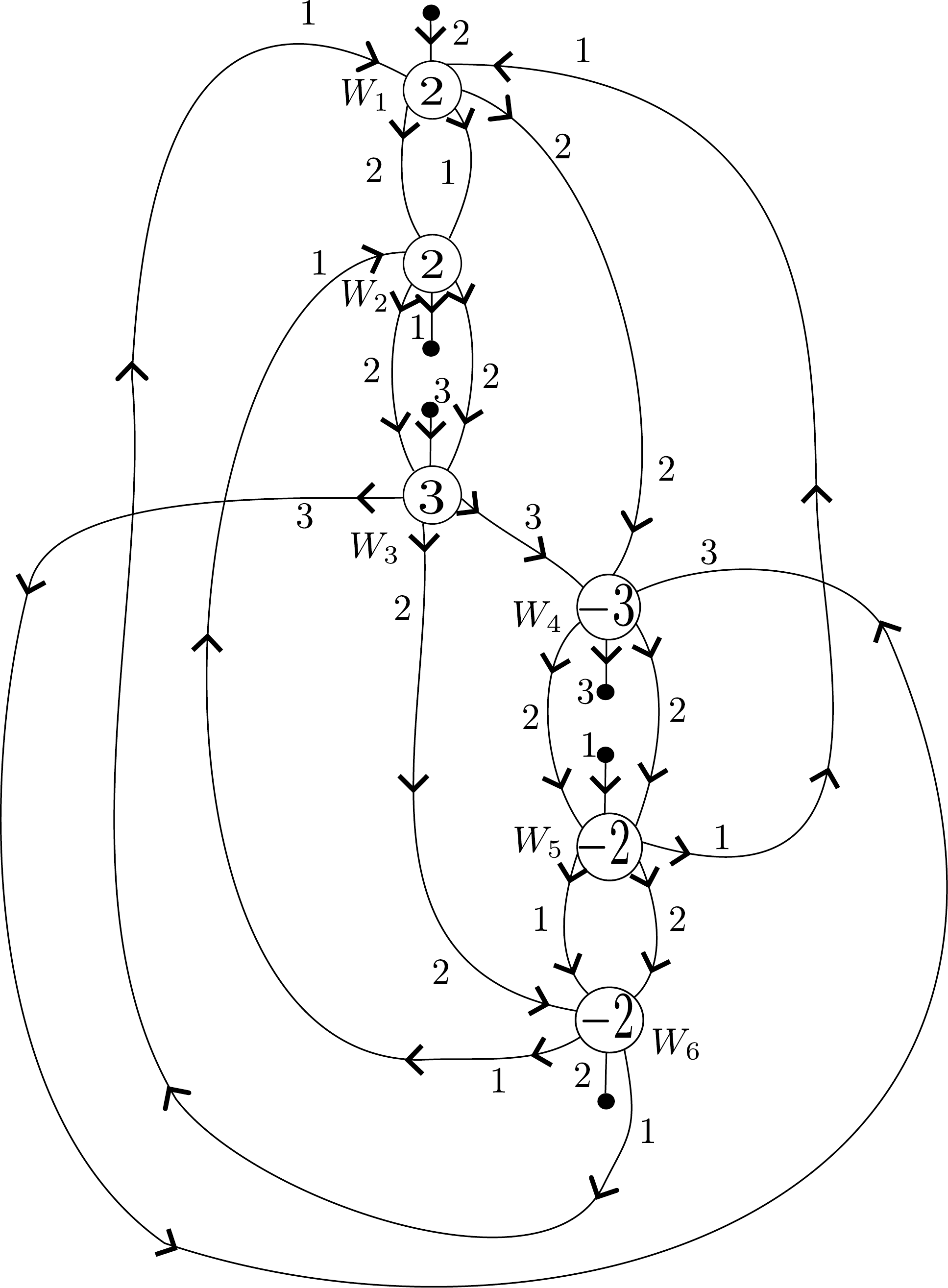}
\end{center}
\caption{Braid chart $\Gamma$.}
\label{fig:braid_chart}
\end{figure}

\begin{table}[H]
\caption{~}
\label{table:1}
$\begin{array}{ |>{\centering\arraybackslash$} p{.15cm} <{$} | >{\centering\arraybackslash$} p{.17cm} <{$} ||>{\centering\arraybackslash$} p{1.8cm} <{$} | >{\centering\arraybackslash$} p{2.2cm} <{$} |>{\centering\arraybackslash$} p{1.9cm} <{$} | >{\centering\arraybackslash$} p{2.0cm} <{$} |>{\centering\arraybackslash$} p{2.2cm} <{$} |>{\centering\arraybackslash$} p{2.2cm} <{$}|| >{\centering\arraybackslash$} p{.3cm} <{$} |}
 \hline
&&&&&&&&\\[-1em]
 y_1 & y_2 & T^{-2}\theta(y_1*y_2, y_1, y_2) & T^{-2}\theta(y_1*y_2, y_2, y_1*y_2) & T^{-3}\theta(y_2, y_1*y_2, y_1) & T^3\theta(y_1, y_1*y_2,y_2)^{-1} & \begin{aligned}[t] &T^2\theta(y_1*y_2, \\&y_1, y_1*y_2)^{-1} \end{aligned} & T^2 \theta(y_1, y_2, y_1*y_2)^{-1} & \prod \\
 \hline
0 & 0 & 1 & 1 & 1 & 1 &1 &1 &1\\
\hline
&&&&&&&&\\[-1em]
0 &1 &uv & 1 & 1 &u^2 v^2 & u^2 v^2 & u^2 v^2 & uv\\
\hline
&&&&&&&&\\[-1em]
0 & 2  &uv & 1 & 1 &u^2 v^2 & u^2 v^2 & u^2 v^2 & uv\\
\hline
&&&&&&&&\\[-1em]
1 & 0 & 1 & uv & uv & 1 & 1 & u^2v^2 & uv\\
\hline
1 &1 & 1 & 1 & 1 & 1 &1 &1 &1\\
\hline
&&&&&&&&\\[-1em]
1 & 2 & uv & 1 & uv & u^2v^2 & 1 & 1 &uv\\
\hline
&&&&&&&&\\[-1em]
2 & 0 & 1 & uv & uv &1 & 1 & u^2 v^2 & uv\\
\hline
&&&&&&&&\\[-1em]
2 & 1 & uv & 1 & uv & u^2 v^2 & 1 & 1 &uv\\
\hline
2 & 2 & 1 &1 & 1& 1& 1& 1 &1\\
\hline
\end{array}$

\end{table}

In particular, consider $X$ to be the twisted quandle $(R_3, (1~2), *)$, where $R_3$ is the diherdal quandle on three elements $\{0,1,2\}$, and $M$ to be $\Lambda_3/\langle T^2-1 \rangle$. As an abelian group $M$ is generated by $1$ and $T$, denoted multiplicatively by $u $ and $v$, respectively. Consider a $3$-cocycle $\theta \in Z^3_{\TBQ, (1)}(X;M)$, where
\begin{equation}\label{eq:theta}
\theta=
(uv)^{\chi(0,1,2) + \chi(0,2,1) + \chi(1,0,1) +\chi(2,0,2) + \chi(1,0,2) + \chi(2,0,1)}.
\end{equation}
On considering all the possible colorings from $Q(S)$ to $R_3$ shown in Table \ref{table:1}, where $y_1=\mathcal{C}(x_1)$ and $y_2=\mathcal{C}(x_2)$, the state-sum invariant of the $2$-twist spun trefoil is
$$\Phi_1(\hat{S})= 3 +6uv.$$
\end{example}

\begin{example}
Now consider the braid chart $\Gamma'$ shown in Figure \ref{fig:braid_chart}, representing a surface braid $S'$ whose closure is the reverse oriented $2$-twist spun trefoil. The fundamental quandle $Q(S')$ (see \cite[Section 11]{MR1990571}) of $S'$ is
\[
\langle x_1, x_2 ~|~ x_2=(x_1*x_2)*x_1, ~x_2=(x_2*x_1)*x_1 \rangle,
\]
where the quandle triples of the white vertices $W_1, \ldots, W_6$ are 
\[
(x_2, x_1*x_2, x_1), ~(x_2, x_1, x_2), ~(x_1, x_2, x_1*x_2), (x_1*x_2, x_2, x_1), (x_2, x_1*x_2, x_2), (x_1*x_2, x_1, x_2), 
\]
respectively. The sign of the white vertices are
\[
\epsilon(W_1)=\epsilon(W_2)=\epsilon(W_3)=-1,~~\epsilon(W_4)=\epsilon(W_5)=\epsilon(W_6)=1.
\]
The numbers in the white vertices represents their corresponding braid index. By Theorem \ref{thm:surface_braid}, for $\theta \in Z^3_{\TBQ, (n)}(X;M)$
\begin{equation}
    \begin{split}
\Phi_n(\hat{S'})=&\sum_{\mathcal{C}} T^{-2n}\theta(\mathcal{C}(x_2), \mathcal{C}(x_1*x_2), \mathcal{C}(x_1))^{-1}\times T^{-2n}\theta(\mathcal{C}(x_2), \mathcal{C}(x_1), \mathcal{C}(x_2))^{-1}\\ &\times T^{-3n} \theta(\mathcal{C}(x_1), \mathcal{C}(x_2), \mathcal{C}(x_1*x_2))^{-1} \times T^{3n}\theta(\mathcal{C}(x_1*x_2), \mathcal{C}(x_2), \mathcal{C}(x_1)) \\ &\times T^{2n} \theta(\mathcal{C}(x_2), \mathcal{C}(x_1*x_2), \mathcal{C}(x_2)) \times T^{2n} \theta(\mathcal{C}(x_1*x_2), \mathcal{C}(x_1), \mathcal{C}(x_2))
\end{split}
\end{equation}

\begin{figure}[H]
\begin{center}

\includegraphics[height=4.5in,width=4in,angle=00]{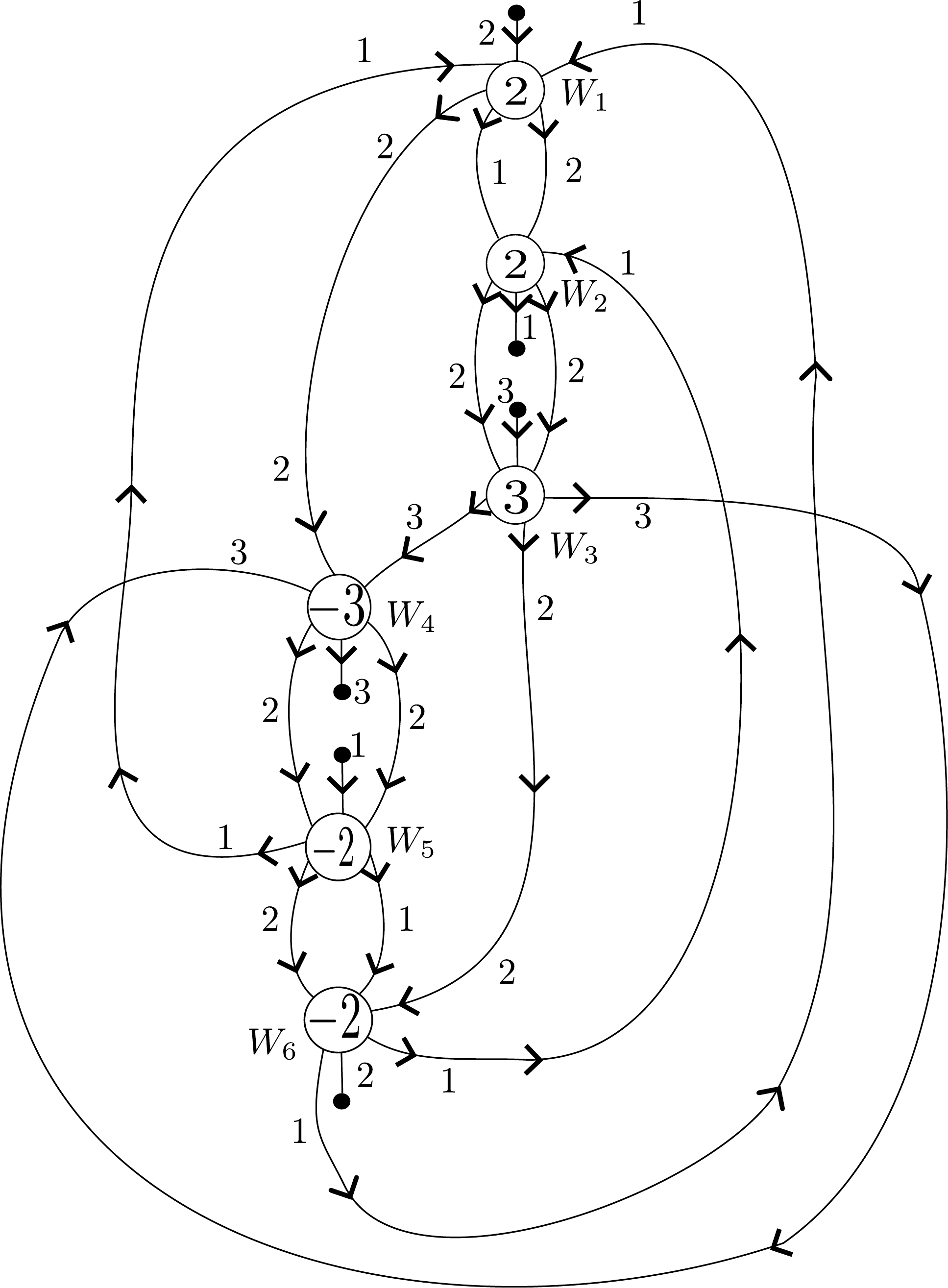}
\end{center}
\caption{Braid chart $\Gamma'$.}
\label{fig:braid_chart_reverse}
\end{figure}

\begin{table}[H]
\caption{~}\label{table:2}
$\begin{array}{ |>{\centering\arraybackslash$} p{.15cm} <{$} | >{\centering\arraybackslash$} p{.17cm} <{$} ||>{\centering\arraybackslash$} p{2cm} <{$} | >{\centering\arraybackslash$} p{2cm} <{$} |>{\centering\arraybackslash$} p{2cm} <{$} | >{\centering\arraybackslash$} p{2.0cm} <{$} |>{\centering\arraybackslash$} p{2.0cm} <{$} |>{\centering\arraybackslash$} p{2.0cm} <{$}|| >{\centering\arraybackslash$} p{.6cm} <{$} |}
 \hline
&&&&&&&&\\[-1em]
 y_1 & y_2 & T^{-2}\theta(y_2, y_1*y_2, y_1)^{-1} &
\begin{aligned}[t]
& T^{-2}\theta(y_2,\\
& y_1, y_2)^{-1}
\end{aligned} & \begin{aligned}[t] &T^{-3}\theta(y_1,\\ &y_2, y_1*y_2)^{-1} \end{aligned} & T^3\theta(y_1*y_2, y_2, y_1)& T^2\theta(y_2, y_1*y_2, y_2) & T^2 \theta(y_1*y_2, y_1, y_2) & \prod \\
 \hline
0 & 0 & 1 & 1 & 1 & 1 &1 &1 &1\\
\hline
&&&&&&&&\\[-1em]
0 & 1 & 1 & u^2 v^2 & u^2 v^2 & 1 & 1 & uv & u^2 v^2\\
\hline
&&&&&&&&\\[-1em]
0 & 2 & 1 & u^2 v^2 & u^2 v^2 & 1 & 1 & uv & u^2 v^2\\
\hline
&&&&&&&&\\[-1em]
1 & 0 & u^2 v^2 & 1 & u^2v^2 & uv & 1 & 1 & u^2 v^2\\
\hline
1 & 1& 1 & 1 & 1 & 1 & 1 & 1 &1\\
\hline
&&&&&&&&\\[-1em]
1 & 2 & u^2v^2 & 1 &1 & uv & uv & uv & u^2 v^2\\
\hline
&&&&&&&&\\[-1em]
2 & 0 &  u^2 v^2 & 1 & u^2v^2 & uv & 1 & 1 & u^2 v^2\\
\hline
&&&&&&&&\\[-1em]
2 & 1 & u^2v^2 & 1 &1 & uv & uv & uv & u^2 v^2\\
\hline
&&&&&&&&\\[-1em]
2 & 2 & 1 & 1 & 1 & 1 & 1 & 1 & 1\\
\hline
\end{array}$
\end{table}
In particular, take $X=(R_3, (1~2), *)$ and $M=\Lambda_3/\langle T^2-1 \rangle$, and $\theta \in Z^3_{\TBQ, (1)}(X;M)$ as in Equation \eqref{eq:theta}. On considering all the possible colorings shown in Table \ref{table:2}, where $y_1=\mathcal{C}(x_1)$ and $y_2=\mathcal{C}(x_2)$, the state-sum invariant of the reverse oriented $2$-twist spun trefoil is
$$\Phi_1(\hat{S'})= 3 +6u^2v^2.$$

\end{example}
\section{Concluding remarks}\label{sec:concluding_remarks}
Analogous to Section \ref{sec:twisted_biquandle_cocycle_invariants_of_classical_knots} and Section \ref{sec:twisted_biquandle_cocycle_invariants_of_knotted_surfaces}, twisted biquandle cocycles can be used to defined invariants for knot diagrams on compact oriented surfaces up to Reidemeister moves and broken surface diagrams in compact oriented $4$-manifolds up to Roseman moves. Here we briefly describe the process for knots on compact surfaces.

Let $K$ be an oriented knot diagram on a compact oriented surface $S$, and a finite twisted biquandle $(X,f,R)$, where the order of $f$ is $p \in \mathbb{Z}_{\geq 0}$. Let $\phi \in Z^2_{\TBQ, (0,n)}(X;M)$, where $M$ is a $\mathbb{Z}[T, T^{-1}]$-module and $\mathcal{C}$ a coloring of $K$ by $(X,f,R)$. The diagram $K$ divides $S$ into regions. Fix a base region denoted by $R_0$, and define $(\mod p)$-{\it Alexander numbering} of the regions (and crossings) as done in Section \ref{sec:twisted_biquandle_cocycle_invariants_of_classical_knots}, where $\mathcal{L}(R_0)=0 \mod p$. If $(\mod p)$-Alexander numbering is not defined, then set the state sum invariant $\Phi(K)$ to be $0$. Otherwise, to each crossing $\tau$, assign a {\it twisted Boltzmann weight} $B_{\T}(\tau, \mathcal{C})=T^{-n\mathcal{L}(\tau)}(\phi(x,y))^{\epsilon(\tau)}$ as done in Section \ref{sec:twisted_biquandle_cocycle_invariants_of_classical_knots}, and define the {state-sum}
\[
\Phi(K)=\sum_{\mathcal{C}}\prod_{\tau}B_{\T}(\tau, \mathcal{C})
\]
Note that $\Phi(K)$ depends on the choice of the base region $R_0$. To overcome this, we consider $\Phi(K)$ up to action of the free abelian group generated by $T$. Thus we have the following result.

\begin{theorem}
The state-sum is well defined up to the action of $\mathbb{Z}=<T>$ for knots on surfaces.
\end{theorem}

\section*{Acknowledgement} 
ME is partially supported by Simons Foundation collaboration grant 712462.  MS is supported by the Fulbright-Nehru postdoctoral fellowship. MS also thanks the Department of Mathematics and Statistics at the University of South Florida for hospitality. The authors thank Masahico Saito and Scott Carter for fruitful discussions.

\bibliography{references1}{}
\bibliographystyle{abbrv}
\end{document}